\documentclass[12pt,a4paper,reqno]{amsart} %Fr o m detta
\usepackage{amssymb}
\usepackage{latexsym}
\usepackage{amsmath}
\usepackage{mathrsfs}
\usepackage[X2,T1]{fontenc}
\usepackage[applemac]{inputenc}
\usepackage{cite}
\usepackage{calc}                   %Detta beh?vs f?r att
\newcommand{\scal}[2]{\langle #1,#2\rangle}
\newcommand{\rr}[1]{\mathbf R^{#1}}

\newcommand{\cc}[1]{\mathbf C^{#1}}
\newcommand{\nm}[2]{\Vert #1\Vert _{#2}}

\newcommand{\sets}[2]{\{ \, #1\, ;\, #2\, \} }
\newcommand{\Sets}[2]{\left \{ \, #1\, ;\, #2\, \right \} }

\newcommand{\cdo}{\, \cdot \, }

\newcommand{\wpr}{{\text{\footnotesize $\#$}}}

\newcommand{\eabs}[1]{\langle #1\rangle}     %%%%%   for <x>
\newcommand{\ON}{\operatorname{ON}}

\newcommand{\vrum}{\vspace{0.1cm}}

\newcommand{\OG}{\operatorname{OG}}

\newcommand{\nn}[1]{{\mathbf N}^{#1}}

\newcommand{\maclH}{\mathcal H}

\newcommand{\maclM}{\mathcal M}
\newcommand{\maclS}{\mathcal S}
\newcommand{\mascB}{\mathscr B}

\newcommand{\mascF}{\mathscr F}
\newcommand{\mascH}{\mathscr H}
\newcommand{\mascI}{\mathscr I}

\newcommand{\mascP}{\mathscr P}
\newcommand{\mascS}{\mathscr S}
\newcommand{\bsySig}{\boldsymbol \Sigma}

\newcommand{\bsycalS}{\boldsymbol {\mathcal S}}

\numberwithin{equation}{section}          %Detta g?r att man f?r
\newtheorem{thm}{Theorem}
\numberwithin{thm}{section}

\newcommand{\rubrik}{}
\newtheorem{prop}[thm]{Proposition}

\newtheorem{lemma}[thm]{Lemma}

\theoremstyle{definition}

\newtheorem{defn}[thm]{Definition}

\theoremstyle{remark}
\newtheorem{rem}[thm]{Remark}

\author{Yuanyuan Chen}

\address{Department of Mathematics,
Linn{\ae}us University, V{\"a}xj{\"o}, Sweden}

\email{yuanyuan.chen@lnu.se}

\author{Michael Signahl}

\address{Department of Mathematical Sciences,
University of Agder, Kristiansand, Norway}

\email{mikael.signal@lnu.se}

\author{Joachim Toft}

\address{Department of Mathematics,
Linn{\ae}us University, V{\"a}xj{\"o}, Sweden}

\email{joachim.toft@lnu.se}

\title{Factorizations and singular value estimates of
operators with Gelfand-Shilov and Pilipovi{\'c} kernels}



\begin{document}

\begin{abstract}
We prove that any linear operator with kernel in a Pilipovi{\'c} or Gelfand-Shilov space
can be factorized by two operators in the same class. We also give links on
numerical approximations for such compositions. We apply these composition rules
to deduce estimates of singular values and establish Schatten-von Neumann properties
for such operators.
%An absent abstained abstract.
\end{abstract}

\maketitle

\par

%%%%%%%%%%%%%%%%%%%%%%%
\section{Introduction}\label{sec0}
%%%%%%%%%%%%%%%%%%%%%%%

\par

%The singular values of linear operators indicate possibilities of obtaining
%finite rank approximations of the operators.
The singular values and their decays are strongly related to possibilities
of obtaining suitable finite rank approximations of the operators. For a
linear and compact operator which acts between Hilbert spaces, the
singular values are the eigenvalues in decreasing order of the
modulus of the operator. If more generally, the linear operator $T$ is
continuous from the quasi-Banach space $\mascB _1$ to (another)
quasi-Banach space $\mascB _2$, then the singular value of order
$j\ge 1$ is given by
\begin{equation}\label{Eq:SingValDef}
\sigma _j(T)= \sigma _j(T,\mascB _1,\mascB _2) \equiv 
\inf \nm {T-T_0}{\mascB _1\to \mascB _2},
\end{equation}
where the infimum is taken over all linear operators $T_0$ from
$\mascB _1$ to $\mascB _2$ of rank at most $j-1$. (See Section \ref{sec1}
for notations.) It follows that $T$ is compact, if and only if $\sigma _j(T)$ decreases to
zero as $j$ tends to infinity, or equivalently, $T$ can be approximated by finite
rank operators with arbitrarily small errors.

\par

In this paper we deduce estimates of $\sigma _j(T)$ when $\mascB _1$ and
$\mascB _2$ stays between small test function spaces, denoted by $\maclH _s(\rr d)$
and $\maclH _{0,s}(\rr d)$, and their (large) duals. The spaces $\maclH _s(\rr d)$
and $\maclH _{0,s}(\rr d)$ depend on the parameter $s\ge 0$ and are obtained by
imposing certain exponential type estimates on the Hermite coefficients of the Hermite
series expansions of the involved functions. More precisely, the set $\maclH _s(\rr d)$
($\maclH _{0,s}(\rr d)$) consists of all
$$
f = \sum _{\alpha }c_\alpha h_\alpha
$$
such that $|c_\alpha |\lesssim e^{-c|\alpha |^{\frac 1{2s}}}$ for some (for every) $c>0$. It
follows that $\maclH _s(\rr d)$ and $\maclH _{0,s}(\rr d)$ increase with $s$, and are
continuously embedded and dense in $\mascS (\rr d)$.

\par

In \cite{To13} the spaces $\maclH _s(\rr d)$ and $\maclH _{0,s}(\rr d)$ and their
duals were characterized in different ways. For example, the images
under the Bargmann transform were given, and it was proved that
$f\in \maclH _s(\rr d)$ ($f\in \maclH _s(\rr d)$), if and only if $f$ satisfies
\begin{equation}\label{Eq:IntHarmPowEst}
|H^Nf(x)| \lesssim h^NN!^{2s}
\end{equation}
for some $h>0$ (for every $h>0$), where $H=H_d$ is the harmonic oscillator $|x|^2-\Delta$
on $\rr d$.
In this context we recall that Pilipovi{\'c} introduced in \cite{Pil2} function spaces whose elements
obey estimates of the form \eqref{Eq:IntHarmPowEst} for certain choices of $s$. For
this reason, we call $\maclH _s(\rr d)$ and $\maclH _{0,s}(\rr d)$ the Pilipovi{\'c}
spaces of Roumieu and Beurling type, respectively, of degree $s\ge 0$ (cf. in \cite{To13}).

\par

In \cite{Pil2}, it is also proved that $\maclH _{s_1}(\rr d)$ and $\maclH _{0,s_2}(\rr d)$
agree with the Gelfand-Shilov spaces $\maclS _{s_1}(\rr d)$ and $\Sigma _{s_2}(\rr d)$,
respectively, when $s_1\ge \frac 12$ and $s_2>\frac 12$, while $\maclH _{0,\frac 12}(\rr d)$
is different from the trivial space $\Sigma _{\frac 12}(\rr d)=\{ 0\}$. The family of Pilipovi{\'c}
spaces therefore contains all Gelfand-Shilov
spaces which are invariant under Fourier transformations.

\par

In Section \ref{sec4} we consider linear operators whose kernels belong to
$\maclH _s(\rr {2d})$. We show that the singular values of such operator
satisfies the estimate
\begin{equation}\label{Eq:StrongSVEst1}
\sigma _k(T,\mascB _1,\mascB _2) \lesssim e^{-ck^{\frac 1{2ds}}}
\end{equation}
for some $c>0$, when $\mascB _j$ stays between $\maclH _s(\rr d)$ and
its dual. If the $\maclH _s$-spaces and their duals are replaced by
$\maclH _{0,s}$-spaces and their duals, then
we also prove that \eqref{Eq:StrongSVEst1} is true for every $c>0$. Furthermore,
if $\maclH _s$-spaces and their duals are replaced by Schwartz spaces and their duals,
then we prove
\begin{equation}\label{Eq:StrongSVEst2}
\sigma _k(T,\mascB _1,\mascB _2) \lesssim \eabs k^{-N}
\end{equation}
for every $N\ge 0$, which should be available in the literature.

\par

These singular-value estimates are based on the fact that the operator classes
here above possess convenient factorization properties, which are deduced
in Section \ref{sec3}. More precisely, an
operator class $\maclM$ is called a \emph{factorization algebra}, if for every
$T\in \maclM$, there exist $T_1,T_2\in \maclM$ such that $T=T_1\circ T_2$.
(In \cite{ToKrNiNo} the term \emph{decomposition algebra} is used instead
of factorization algebra.) Evidently, $\mathscr L(\mascB )$, the set of
continuous linear operators on the quasi-Banach space $\mascB$ is a
factorization algebra, since we may choose $T_1$ as the identity operator and
$T_2=T$. A more challenging situation appears when $\maclM$ does not
contain the identity operator, and in this situation it is easy to find operator
classes which are not factorization algebras. For example, any Schatten-von
Neumann class of finite order is not a factorization algebra.

\par

If $\mascB$ above is a Hilbert space and $\maclM$ is the set of compact
operators on $\mascB$, then it follows by an application of the spectral
theorem that $\maclM$ is a factorization algebra. It is also well-known
that the set of linear operators with kernels in the Schwartz space is a
factorization algebra (see e.{\,}g. \cite{Beals, Kl, Sad, ToKrNiNo, Vo}).
Furthemore, similar facts hold true for the set of operators with kernels
in a fixed Gelfand-Shilov space (cf. \cite{ToKrNiNo}).

\par

In Section \ref{sec3} we extend the latter property such that all Pilipovi{\'c}
spaces are included. That is, we prove that the set of operators with kernels
in a fixed (but arbitrarily chosen) Pilipovi{\'c} space is a factorization algebra.
%(cf. \cite{Pil2,To13}).

\par

Since the singular values of the operators under considerations either satisfy
conditions of the form \eqref{Eq:StrongSVEst1} or \eqref{Eq:StrongSVEst2}
for every $N\ge 0$, it follows that the sequence $\{ \sigma _j(T)\} _{j=1}^\infty$
belongs to $\ell ^p$ for every $p>0$. This implies that any such operator is a
Schatten-von Neumann operator of degree $p$ for every $p>0$.

\par

Here we remark that the latter conclusions in the Gelfand-Shilov situation, were
deduced in \cite{ToKrNiNo} in slight different ways, which enables to replace
the quasi-Banach spaces $\mascB _1$ and $\mascB _2$ by convenient Hilbert
spaces. The main property behind the latter reduction concerns
\cite[Proposition 3.8]{To12}, where it is proved that if $s\ge \frac 12$ and
$$
\maclH _s(\rr d)\subseteq \mascB \subseteq \maclH _s'(\rr d),
$$
then there are Hilbert spaces $\mascH _1$ and $\mascH _2$ such that
$$
\maclH _s(\rr d)\subseteq \mascH _1 \subseteq \mascB \subseteq
\mascH _2\subseteq \maclH _s'(\rr d).
$$
The Schatten-von Neumann properties are then obtained in straight-forward ways
by the factorization properties in combination with the exact formulas, for
Hilbert-Schmidt norms of operators acting between Hilbert spaces.

\par

Our investigations also include analysis of operators with kernels in
$\maclH _{\flat _\sigma}$, $\maclH _{0,\flat _\sigma}$, $\sigma >0$, or their
duals. These spaces were carefully investigated in \cite{To13} and are defined
by imposing conditions of the form $h^{|\alpha |}(\alpha !)^{\frac 1{2\sigma}}$ on
the Hermite coefficients of the involved functions. In \cite{To13}, these spaces
are characterized in different ways. For example, it is here proved that the
Bargmann transform is bijective from $\maclH _{\flat _\sigma}(\rr d)$ to the
set of all entire functions $F$ on $\cc d$ such that
$$
|F(z)|\lesssim e^{c|z|^{\frac {2\sigma}{\sigma +1}}}
$$
for some constant $C>0$.

\par

In Section \ref{sec2} we deduce kernel theorems for operators with kernels
in these spaces, or related distribution spaces. In Section \ref{sec3} we
show certain factorization properties of operators with kernels in
$\maclH _{\flat _\sigma}$ or in $\maclH _{0,\flat _\sigma}$. These factorization
results are slightly weaker compared to what is deduced for operators with kernels in
$\maclH _s$ and $\maclH _{0,s}$ when $s\ge 0$ is real.

\par

Finally we apply these factorization properties in Section \ref{sec4}, to obtain singular
value decompositions for operators with kernels in $\maclH _{\flat _\sigma}$
or in $\maclH _{0,\flat _\sigma}$. In particular we show that if $T$ is an operator on
$L^2(\rr d)$ with kernel in $\maclH _{\flat _\sigma}(\rr {2d})$, then the singular
values of $T$ satisfy
$$
\sigma _j(T)\lesssim h^j(j!)^{-\frac 1{2\sigma d}},
$$
for some constant $h>0$.
%
%
%
%
%
%
%In \cite{ToKrNiNo} these Schatten-von Neumann results were deduced in the case of
%Gelfand-Shilov and Schwartz kernels by using the factorization technique above in
%combination with a reduction techniques, which enables to replace the quasi-Banach
%spaces $\mascB _1$ and $\mascB _2$ by convenient Hilbert spaces. The main
%property behind the latter reduction concerns \cite[Proposition 3.8]{To12}, where
%it is proved that if $s\ge \frac 12$ and
%$$
%\maclH _s(\rr d)\subseteq \mascB \subseteq \maclH _s'(\rr d),
%$$
%then there are Hilbert spaces $\mascH _1$ and $\mascH _2$ such that
%$$
%\maclH _s(\rr d)\subseteq \mascH _1 \subseteq \mascB \subseteq \mascH _2\subseteq
%\maclH _s'(\rr d).
%$$
%In \cite{To12} it is also proved that the same holds true if in addition $s>\frac 12$ and
%$\maclH _s$ and its dual is replaced by $\maclH _{0,s}$ or $\mascS $, and their duals. 
%In Section \ref{sec4} we extend for completeness the latter result such that it also
%includes the case when $\maclH _s$ or $\maclH _{0,s}$ can be any
%Pilipovi{\'c} space, i.{\,}e. $s$ is allowed to be smaller than $\frac 12$.

\par

%%%%%%%%%%%%%%%%%%%%%%%
\section{Preliminaries}\label{sec1}
%%%%%%%%%%%%%%%%%%%%%%%

\par

In this section we recall some basic facts. We start by discussing
Pilipovi{\'c} % Gelfand-Shilov
spaces and their properties. Thereafter we consider
suitable spaces of formal Hermite series expansions, and discuss
their links with Pilipovi{\'c} spaces.

%
%In this section we recall some basic facts. We start by discussing
%Gelfand-Shilov spaces and their properties. Thereafter we recall
%the definition of Pilipovi{\'c} spaces and some of their properties.
%Then we recall some facts on modulation spaces. Finally
%we recall the Bargmann transform and some of its
%mapping properties, and introduce suitable classes of entire functions on
%$\cc d$.

\par

%%% %
\subsection{The Pilipovi{\'c} spaces}\label{subsec1.1}
%%% %
We start to consider spaces which are obtained by suitable
estimates of Gelfand-Shilov or Gevrey type when using powers of the harmonic
oscillator $H=|x|^2-\Delta$, $x\in \rr d$. %(Cf. \cite{GrPiRo1,Pil1,Pil2}.)

\par

Let $h>0$, $s\ge 0$ and let $\bsycalS _{\! h,s}(\rr d)$ be the set of all
$f\in C^\infty (\rr d)$ such that
\begin{equation}\label{GFHarmCond}
\nm f{\bsycalS _{\! h,s}}\equiv \sup _{N\ge 0}
\frac {\nm{H^Nf}{L^\infty}}{h^N(N!)^{2s}}<\infty .
\end{equation}
Then $\bsycalS _{\! h,s}(\rr d)$ is a Banach space. If $s>0$, then
$\bsycalS _{\! h,s}(\rr d)$ contains all Hermite functions. Furthermore,
if $s=0$, and $\alpha \in \nn d$ satisfies $2|\alpha |+d\le h$,
then $h_\alpha \in \bsycalS _{\! h,s}(\rr d)$.

\par

We let
$$
\bsySig _s(\rr d) \equiv \bigcap _{h>0}\bsycalS _{\! h,s}(\rr d)
\quad \text{and}\quad
\bsycalS _{\! s}(\rr d) \equiv \bigcup _{h>0}\bsycalS _{\! h,s}(\rr d),
$$
and equip these spaces by projective and inductive limit topologies,
respectively, of $\bsycalS _{\! h,s}(\rr d)$, $h>0$. (Cf. \cite{GrPiRo,Pil1,Pil2,To13}.)

\par

In \cite{Pil1,Pil2}, Pilipovi{\'c} proved that if $s_1\ge \frac 12$ and $s_2>\frac 12$, then
$\bsycalS _{\! s_1}(\rr d)$ and $\bsySig _{s_2}(\rr d)$ agree with the
Gelfand-Shilov spaces $\maclS _{s_1}(\rr d)$ and $\Sigma _{s_2}(\rr d)$, respectively.
(See e.{\,}g. \cite{To13} for notations.) 
On the other hand,  $\maclS _{s_1}(\rr d)$ and $\Sigma _{s_2}(\rr d)$ are trivially equal
to $\{ 0 \}$ when $s_1<\frac 12$ and $s_2\le \frac 12$, while any
Hermite function $h_\alpha$ fulfills \eqref{GFHarmCond} for every $h>0$,
when $0<s\le \frac 12$.

\par

Hence,
%%%
%\begin{alignat*}{3}
%\bsycalS _{\! s_1}(\rr d) &= \maclS _{s_1}(\rr d),&\quad
%\bsySig _{s_2}(\rr d) &= \Sigma _{s_2}(\rr d),& \qquad s_1 &\ge \frac 12,\ s_2 > \frac 12
%\intertext{and}
%\bsycalS _{\! s_1}(\rr d) &\neq \maclS _{s_1}(\rr d)= \{ 0\},
%&\quad \bsySig _{s_2}(\rr d) &\neq \Sigma _{s_2}(\rr d)=\{ 0\} , & \qquad s_1 &<\frac 12,\ 0<s_2
%\le \frac 12.
%\end{alignat*}
%%%
%%
\begin{alignat*}{3}
\bsycalS _{\! s_1} &= \maclS _{s_1},&\quad
\bsySig _{s_2} &= \Sigma _{s_2},& \qquad s_1 &\ge \frac 12,\quad \phantom{0<}s_2 > \frac 12
\intertext{and}
\bsycalS _{\! s_1} &\neq \maclS _{s_1}= \{ 0\},
&\quad \bsySig _{s_2}&\neq \Sigma _{s_2} = \{ 0\} , & \qquad s_1 &<\frac 12,\quad 0<s_2
\le \frac 12.
\end{alignat*}
The space $\bsySig _s(\rr d)$ is called the \emph{Pilipovi{\'c} space (of Beurling type) of
order $s\ge 0$ on $\rr d$}.
Similarly, $\bsycalS _{\! s}(\rr d)$ is called the \emph{Pilipovi{\'c} space
(of Roumieu type) of order $s\ge 0$ on $\rr d$}.

\par

The dual spaces of $\bsycalS _{\! h,s}(\rr d)$, $\bsySig _s(\rr d)$
and $\bsycalS _{\! s}(\rr d)$ are denoted by $\bsycalS _{\! h,s}'(\rr d)$,
$\bsySig _s'(\rr d)$ and $\bsycalS _{\! s}'(\rr d)$, respectively.
We have
$$
\bsySig _s'(\rr d) = \bigcup _{h>0} \bsycalS _{\! h,s}'(\rr d)
$$
when $s>0$ and
$$
\bsycalS _{\! s}'(\rr d) = \bigcap _{h>0} \bsycalS _{\! h,s}'(\rr d)
$$
when $s\ge 0$, with inductive respective projective limit topologies
of $\bsycalS _{\! h,s}'(\rr d)$, $h>0$ (cf. \cite{To13}).

\par

%%%
\subsection{Spaces of Hermite series expansions}\label{subsec1.3}
%%%

\par

Next we recall the definitions of topological vector spaces of
Hermite series expansions, given in \cite{To13}. As in \cite{To13},
it is convenient to use the sets $\mathbf R_\flat$ and
$\overline {\mathbf R_\flat}$ when indexing our spaces.
%
%It is convenient to consider
%slight larger sets $\mathbf R_\flat$ and $\overline {\mathbf R_\flat}$ of
%the set of all positive real numbers, $\mathbf R_+$.

\par

\begin{defn}
The sets $\mathbf R_\flat$ and $\overline {\mathbf R_\flat}$ are given by
$$
{\textstyle{\mathbf R_\flat = \mathbf R_+ \underset{\sigma >0}{\textstyle{\bigcup}}
\{ \flat _\sigma \} }}
\quad \text{and}\quad
{\textstyle{\overline {\mathbf R_\flat} = \mathbf R_\flat \bigcup \{ 0 \} }}.
$$

\par

Moreover, beside the usual ordering in $\mathbf R$, the elements $\flat _\sigma$
in $\mathbf R_\flat$ and $\overline {\mathbf R_\flat}$ are ordered by
the relations $x_1<\flat _{\sigma _1}<\flat _{\sigma _2}<x_2$, when $\sigma _1<\sigma _2$,
$x_1<\frac 12$ and $x_2\ge \frac 12$
are real. %Furthermore, $\flat _\infty \equiv \frac 12$.
%
% Here I have removed about $\flat _\infty$. /JT
%
%
%
\end{defn}

\par

\begin{defn}\label{DefSeqSpaces}
Let $p\in (0,\infty ]$, $s\in {\mathbf R_\flat}$, $r\in \mathbf R$, $\vartheta$
be a weight on $\nn d$, and let
$$
\vartheta _{r,s}(\alpha )\equiv
\begin{cases}
e^{r|\alpha |^{\frac 1{2s}}}, & \text{when}\quad s\in \mathbf R_+,
\\[1ex]
r^{|\alpha |}(\alpha !)^{\frac 1{2\sigma}}, & \text{when}\quad s = \flat _\sigma ,
\quad \qquad \alpha \in \nn d.
\end{cases}
$$
Then,
\begin{enumerate}
\item $\ell _0' (\nn d)$ is the set of all sequences $\{c_\alpha \} _{\alpha \in \nn d}
\subseteq \mathbf C$ on $\nn d$;

\vrum

\item $\ell _{0,0}(\nn d)\equiv \{ 0\}$, and $\ell _0(\nn d)$ is the set of all sequences
$\{c_\alpha \} _{\alpha \in \nn d}\subseteq \mathbf C$ such that $c_\alpha \neq 0$
for at most finite numbers of $\alpha$;

\vrum

\item $\ell ^p_{[\vartheta ]}(\nn d)$ is the quasi-Banach space which consists of
all sequences $\{ c_\alpha \} _{\alpha \in \nn d} \subseteq \mathbf C$
such that
$$
\nm {\{ c_\alpha \} _{\alpha \in \nn d} }{\ell ^p_{[\vartheta ]}}\equiv
\nm {\{ c_\alpha \vartheta (\alpha )\} _{\alpha \in \nn d} }{\ell ^p}
$$
is finite;

\vrum

\item $\ell _{0,s}(\nn d)\equiv \underset {r>0}\bigcap \ell ^p_{[\vartheta _{r,s}]}(\nn d)$
and $\ell _s(\nn d)\equiv \underset {r>0}\bigcup \ell ^p_{[\vartheta _{r,s}]}(\nn d)$, with
projective respective inductive limit topologies of $\ell ^p_{[\vartheta _{r,s}]}(\nn d)$
with respect to $r>0$;

\vrum

\item $\ell _{0,s}'(\nn d)\equiv \underset {r>0}\bigcup
\ell ^p_{[1/\vartheta _{r,s}]}(\nn d)$ and $\ell _s'(\nn d)\equiv \underset {r>0}
\bigcap \ell ^p_{[1/\vartheta _{r,s}]}(\nn d)$, with
inductive respective projective limit topologies of $\ell ^p_{[1/\vartheta _{r,s}]}(\nn d)$
with respect to $r>0$.
\end{enumerate}
\end{defn}

\par

Let $p\in (0,\infty ]$, and let $\Omega _N$ be the set of all
$\alpha \in \nn d$ such that $|\alpha |\le N$. Then the
topology of $\ell _0(\nn d)$ is defined by the inductive
limit topology of the sets
$$
\Sets {\{ c_\alpha \} _{\alpha \in \nn d} \in \ell _0'(\nn d)}{c_\alpha =0\
\text{when}\ \alpha \neq \Omega _N}
$$
with respect to $N\ge 0$, and whose topology
is given through the quasi-norms
\begin{equation}\label{SemiNormsEllSpaces}
\{ c_\alpha \} _{\alpha \in \nn d}\mapsto \nm {\{ c_\alpha \}
_{|\alpha |\le N} }{\ell ^p(\Omega _N)},
\end{equation}
Since any two quasi-norms on a finite-dimensional vector space are equivalent,
it follows that these topologies are independent of $p$.
Furthermore, the topology of
$\ell _0' (\nn d)$ is defined by the quasi-semi-norms \eqref{SemiNormsEllSpaces}.
It follows that $\ell _0'(\nn d)$ is a Fr{\'e}chet space, and that the topology
is independent of $p$.

\par

Next we introduce spaces of formal Hermite series expansions
\begin{equation}\label{Hermiteseries}
f=\sum _{\alpha \in \nn d}c_\alpha h_\alpha ,\quad \{c_\alpha \}
_{\alpha \in \nn d}\in \ell _0' (\nn d).
\end{equation}
which correspond to
\begin{equation}\label{ellSpaces}
\ell _{0,s}(\nn d),\quad \ell _s(\nn d),\quad \ell ^p_{[\vartheta ]}(\nn d),
\quad \ell _s'(\nn d)\quad \text{and}\quad \ell _{0,s}'(\nn d).
\end{equation}
For that reason we consider the mappings
\begin{equation}\label{T1Map}
T : \,  
\{ c_\alpha \} _{\alpha \in \nn d} \mapsto \sum _{\alpha \in \nn d}
c_\alpha h_\alpha 
\end{equation}
between sequences and formal Hermite series expansions.

\par

\begin{defn}\label{DefclHclASpaces}
Let $p\in (0,\infty ]$, $\vartheta$ be a weight on $\nn d$, and let
$s\in \overline{\mathbf R_\flat}$.
\begin{itemize}
\item the spaces
\begin{equation}\label{clHSpaces}
\maclH _{0,s}(\rr d),\quad \maclH _s(\rr d),\quad \maclH ^p_{[\vartheta ]}(\rr d),
\quad \maclH _s'(\rr d)\quad \text{and}\quad \maclH _{0,s}'(\rr d)
\end{equation}
are the images of $T$ in \eqref{T1Map} under corresponding spaces in
\eqref{ellSpaces}.
Furthermore, the topologies of the spaces in \eqref{clHSpaces}
are inherited from corresponding spaces in \eqref{ellSpaces}.

\vrum

\item the quasi-norm $\nm f{\maclH ^p_{[\vartheta ]}}$ of $f\in \maclH _0'(\rr d)$, 
is given by $\nm {\{ c_\alpha \} _{\alpha \in \nn d}}{\ell ^p_{[\vartheta ]}}$,
when $f$ is given by \eqref{Hermiteseries}.
\end{itemize}
\end{defn}

\par

By the definitions it follows that the inclusions
\begin{multline}\label{inclHermExpSpaces}
\maclH _0(\rr d)\subseteq \maclH _{0,s}(\rr d)\subseteq
\maclH _{s}(\rr d) \subseteq \maclH _{0,t}(\rr d)
\\[1ex]
\subseteq \mascS (\rr d) \subseteq \mascS '(\rr d) \subseteq
\maclH _{0,t}'(\rr d) \subseteq \maclH _{s}'(\rr d)
\\[1ex]
\subseteq \maclH _{0,s}'(\rr d)\subseteq \maclH _0'(\rr d),
\quad \text{when}\ s,t\in \mathbf R_\flat ,\ s<t
\end{multline}
hold true.

\par

\begin{rem}\label{Rem:LinkEllHs}
By the definition it follows that $T$ in \eqref{T1Map} is a homeomorphism
between any of the spaces in \eqref{ellSpaces} and corresponding space
in \eqref{clHSpaces}.
\end{rem}

\par

The next result shows that the spaces in Definition \ref{clHSpaces}
essentially agrees with the Pilipovi{\'c} spaces. We refer to \cite{To13}
for the proof.

\par

\begin{prop}
Let $0\le s\in \mathbf R$. Then $\maclH _{0,s}(\rr d) = \bsySig _s(\rr d)$
and $\maclH _s(\rr d) = \bsycalS _s(\rr d)$.
\end{prop}

\par

\begin{rem}\label{Rem:GelfandTripp}
Let $T$ be given by \eqref{T1Map}. Then
\begin{equation*} %\label{G-tripple1}
\begin{alignedat}{2}
\big ( \ell _s(\nn d), \ell ^2(\nn d),\ell _s'(\nn d) \big ) \ & \overset
{\overset{\scriptstyle{T}}{}} {\rightarrow } &\
\big ( \maclH _s(\rr d), L^2(\rr d),\maclH _s'(\rr d) \big ),\ s &\ge 0,
\\[1ex]
\big ( \ell _{0,s}(\nn d), \ell ^2(\nn d),\ell _{0,s}'(\nn d) \big ) \ & \overset
{\overset{\scriptstyle{T}}{}} {\rightarrow } \ &
\big ( \maclH _{0,s}(\rr d), L^2(\rr d),\maclH _{0,s}'(\rr d) \big ),\ s &>0,
\\[1ex]
\big ( \ell _{[\vartheta]}^p(\nn d), \ell ^2(\nn d),\ell _{[1/\vartheta]}^{p'}(\nn d) \big ) \ & \overset
{\overset{\scriptstyle{T}}{}} {\rightarrow } \ &
\big ( \maclH _{[\vartheta ]}^p(\rr d), L^2(\rr d),\maclH _{[1/\vartheta ]}^{p'}(\rr d) \big ),\
p &\in [1,\infty )
\end{alignedat}
\end{equation*}
are isometric bijections between Gelfand triples. (Cf. e.{\,}g. Section 2 in \cite{To13}.)
\end{rem}

\par

Finally, in Section \ref{sec5} we apply the results from the first sections to obtain certain
characterizations of operators with kernels in $\maclH _s$ and $\maclH _{0,s}$.

%Next we consider classes of matrices which are related to Pilipovi{\'c} kernels.
%
%In what follows we let
%$J_1$ and $J_2$ be index sets, $A$ be the matrix $(a(j,k))_{(j,k)
%\in \nn {d_2+d_1}}$, $p,q\in
%(0,\infty ]$, and $\vartheta$ a map from $\nn {d_2+d_1}$
%to $\mathbf R_+$. Then let
%%%
%\begin{equation}\label{haomegadef}
%h_{A,p,\omega }(k) \equiv \nm {a(\cdo ,k)\vartheta (\cdo ,k)}{\ell ^p},
%\end{equation}
%%%

%\par
%
%\begin{defn}\label{matrixset1}
%Let $0<p,q\le \infty$, $J_1$ and $J_2$ be index sets and let $\vartheta$ be a map
%from $J\equiv \nn {d_2+d_1}$ to $\mathbf R_+$.
%\begin{enumerate}
%\item The set $\mathbb U_0(J)$ consists of matrices $(a(j,k))_{(j,k)\in J}$
%such that at most finite numbers of $a(j,k)$ are non-zero;
%
%\vrum
%
%\item The set $\mathbb U^p(\vartheta ,J)$
%consists of all matrices $A=(a(j,k))_{(j,k)\in J}$ such that
%$$
%\nm A{\mathbb U^p(\vartheta ,J)} \equiv \nm { a
%}{\ell ^p_{(\vartheta )}(J)}
%$$
%is finite.
%\end{enumerate}
%\end{defn}
%
%\par
%
%By Remark \ref{Rem:LinkEllHs} and the fact that $h_{\alpha ,\beta}
%=h_\alpha \otimes h_\beta$, when $\alpha$ and $\beta$ are multi-indices,
%it follows that if $J_1=$

\par

%%%%%%%%%%%%%%%%%%%%%%%%%%%%
\section{Kernel theorems}\label{sec2}
%%%%%%%%%%%%%%%%%%%%%%%%%%%%

\par

In this section we deduce suitable kernel theorems for operators between
Pilipovi{\'c} spaces and their duals.
Since the spaces
under considerations can in convenient ways be formulated in terms of
Hermite series expansions, we may easily reduce ourselves to kernel
results for matrix operators, in similar ways as in e.{\,}g. \cite{ReSi}. We
therefore begin with the latter case.

\par

\begin{prop}\label{Prop:DiscreteKernelThm}
Let $\vartheta _k$ be weight functions on
$\nn {d_k}$, $k=1,2$, $\vartheta (\alpha ,\beta )=
\vartheta _1(\beta )^{-1}\vartheta _2(\alpha )$,
and let $T$ be a linear and continuous map from $\ell ^1_{[\vartheta _1]}(\nn {d_1})$
to $\ell ^\infty _{[\vartheta _2]}(\nn {d_2})$. Then the following is true:
\begin{enumerate}
\item If $A\in \mathbb \ell ^\infty _{[\vartheta ]}(\nn {d_2+d_1})$, then the map
$f\mapsto A\cdot f$ from $\ell _0(\nn {d_1})$ to $\ell _0'(\nn {d_2})$ extends uniquely to
a linear and continuous map from $\ell ^1_{[\vartheta _1]}(\nn {d_1})$
to $\ell ^\infty _{[\vartheta _2]}(\nn {d_2})$;

\vrum

\item there is a unique element $A\in \ell ^\infty _{[\vartheta ]}(\nn {d_2+d_1})$
such that $Tf =A\cdot f$ for every $f\in \ell ^1_{[\vartheta _1]}(\nn {d_1})$. Furthermore,
\begin{equation}\label{Eq:TANormsIdent}
\nm T{\ell ^1_{[\vartheta _1]}(\nn {d_1})\to \ell ^\infty _{[\vartheta _2]}(\nn {d_2})}
=
\nm A{\ell ^\infty _{[\vartheta ]}}.
\end{equation}
\end{enumerate}
\end{prop}

\par

\begin{proof}
The assertion (1) follows by straight-forward estimates and is left for the reader.

\par

(2) Evidently, for some unique
$$ %\nn {d_2+d_1}
A= (a_{\alpha ,\beta})_{(\alpha , \beta )\in \nn {d_2+d_1}}\in \ell _0'(\nn {d_2+d_1}),
$$
$Tf =A\cdot f$ holds for every $f\in \ell _0(\nn {d_1})$. Moreover, let
\begin{align*}
\Omega _1 &= \sets{f_1\in \ell _0(\nn {d_1})}{\nm {f_1}{\ell ^1_{[\vartheta _1]}}\le 1}
\intertext{and}
\Omega _2 &= \sets{f_2\in \ell _0(\nn {d_1})}{\nm {f_2}{\ell ^1_{[1/\vartheta _2]}}\le 1}.
\end{align*}
Since $\ell _0(\nn {d_1})$ and $\ell _0(\nn {d_2})$ are dense in
$\ell ^1_{[\vartheta _1]}(\nn {d_1})$ and $\ell ^1_{[1/\vartheta _2]}(\nn {d_2})$,
respectively, we obtain
\begin{multline*}
\nm T{\ell ^1_{[\vartheta _1]}(\nn {d_1})\to \ell ^\infty _{[\vartheta _2]}(\nn {d_2})}
=
\sup _{f_1\in \Omega _1}\sup _{f_2\in \Omega _2} |(A\cdot f_1,f_2)_{\ell ^2}|
\\[1ex]
=
\sup _{f_1\in \Omega _1}\nm {A\cdot f_1}{\ell ^\infty _{[\vartheta _2]}}
=
\sup _{\beta \in \nn {d_1}}\sup _{\alpha \in \nn {d_2}}|a_{\alpha ,\beta}\vartheta
_1(\beta )^{-1}\vartheta _2(\alpha )|
=
\nm A{\ell ^\infty _{[\vartheta ]}},
\end{multline*}
which gives (2).
\end{proof}

\par

By the links between $\maclH ^p_{[\vartheta _k]}(\rr {d_k})$ and
$\maclH ^p_{[\vartheta ]}(\rr {d_2}\times \rr {d_1})$, and
$\ell ^p_{[\vartheta _k]}(\nn {d_k})$ and
$\ell ^p_{[\vartheta ]}(\nn {d_2}\times \nn {d_1})$, respectively,
the previous proposition immediately gives the following.
(Cf. Remark \ref{Rem:GelfandTripp}.)

\par

\begin{prop}\label{Prop:HermiteKernelThm}
Let $\vartheta _k$ be weight functions on
$\nn {d_k}$, $k=1,2$, $\vartheta (\alpha _2,\beta )= \vartheta
_1(\beta )^{-1}\vartheta _2(\alpha _2)$,
and let $T$ be a linear and continuous map from $\maclH ^1_{[\vartheta _1]}(\rr {d_1})$
to $\maclH ^\infty _{[\vartheta _2]}(\rr {d_2})$. Then the following is true:
\begin{enumerate}
\item If $K\in \maclH ^\infty _{[\vartheta ]}(\rr {d_2}\times \rr {d_1})$, then the map
\begin{equation}\label{Kmap}
f\mapsto \big ( x_2\mapsto \scal {K(x_2,\cdo)}f \big )
\end{equation}
from $\maclH _0(\rr {d_1})$ to $\maclH _0'(\rr {d_2})$ extends uniquely to
a linear and continuous map from $\maclH ^1_{[\vartheta _1]}(\rr {d_1})$
to $\maclH ^\infty _{[\vartheta _2]}(\rr {d_2})$;

\vrum

\item there is a unique element $K\in \maclH ^\infty _{[\vartheta ]}(\rr {d_2}\times \rr {d_1})$
such that
\begin{equation}\label{Kmap2}
Tf =\big ( x_2\mapsto \scal {K(x_2,\cdo)}f \big )
\end{equation}
for every $f\in \maclH ^1_{[\vartheta _1]}(\rr {d_1})$. Furthermore,
\begin{equation}\label{Eq:TANormsIdent2}
\nm T{\maclH ^1_{[\vartheta _1]}(\rr {d_1})\to \maclH ^\infty _{[\vartheta _2]}(\rr {d_2})}
=
\nm K{\maclH ^\infty _{[\vartheta ]}}.
\end{equation}
\end{enumerate}
\end{prop}

\par

We have now the following kernel results.

\par

%\begin{thm}\label{Thm:KernelsEasyDirection}
%Let $K\in \maclH _0'(\rr {d_2}\times \rr {d_1})$, $s\ge 0$ be real and let
%$T$ be the linear and continuous map from $\maclH _0(\rr {d_1})$ to
%$\maclH _0'(\rr {d_2})$, given by \eqref{Kmap}. Then the following is true:
%\begin{enumerate}
%\item if $K\in \maclH _s(\rr {d_2}\times \rr {d_1})$, then $T$ extends uniquely
%to linear and continuous mappings from $\maclH _s'(\rr {d_1})$ to
%$\maclH _s(\rr {d_2})$;
%
%\vrum
%
%\item if $K\in \maclH _{0,s}(\rr {d_2}\times \rr {d_1})$, then $T$ extends uniquely
%to linear and continuous mappings from $\maclH _{0,s}'(\rr {d_1})$ to
%$\maclH _{0,s}(\rr {d_2})$;
%
%\vrum
%
%\item if $K\in \mascS (\rr {d_2}\times \rr {d_1})$, then $T$ extends uniquely
%to linear and continuous mappings from $\mascS '(\rr {d_1})$ to
%$\mascS (\rr {d_2})$;
%
%\vrum
%
%\item if $K\in \mascS '(\rr {d_2}\times \rr {d_1})$, then $T$ extends uniquely
%to linear and continuous mappings from $\mascS (\rr {d_1})$ to
%$\mascS '(\rr {d_2})$;
%
%\vrum
%
%\item if $K\in \maclH _{0,s}'(\rr {d_2}\times \rr {d_1})$, then $T$ extends uniquely
%to linear and continuous mappings from $\maclH _{0,s}(\rr {d_1})$ to
%$\maclH _{0,s}'(\rr {d_2})$;
%
%\vrum
%
%\item if $K\in \maclH _s'(\rr {d_2}\times \rr {d_1})$, then $T$ extends uniquely
%to linear and continuous mappings from $\maclH _s(\rr {d_1})$ to
%$\maclH _s'(\rr {d_2})$.
%\end{enumerate}
%\end{thm}

\par

\begin{thm}\label{Thm:KernelsDifficultDirection}
Let $s\in \overline{\mathbf R_\flat}$, and let $T$ be the linear and continuous map 
from $\maclH _0(\rr {d_1})$ to $\maclH _0'(\rr {d_2})$, given by
\eqref{Kmap}. Then the following is true:
%
%
%$K\in \maclH _0'(\rr {d_2}\times \rr {d_1})$, $s\ge 0$ be real and let
%. Then the following is true:
%
%
%
\begin{enumerate}
\item if $T$ is a linear and continuous map from $\maclH _s'(\rr {d_1})$ to
$\maclH _s(\rr {d_2})$, then there is
$K\in \maclH _s(\rr {d_2}\times \rr {d_1})$ such that \eqref{Kmap2} holds true;

\vrum

\item if $T$ is a linear and continuous map from $\maclH _s(\rr {d_1})$ to
$\maclH _s'(\rr {d_2})$, then there is
$K\in \maclH _s'(\rr {d_2}\times \rr {d_1})$ such that \eqref{Kmap2} holds true.
\end{enumerate}

\par

The same holds true if the $\maclH _s$ and $\maclH _s'$ spaces are replaced
by $\maclH _{0,s}$ and $\maclH _{0,s}'$ spaces, respectively, or by
$\mascS$ and $\mascS '$ spaces, respectively.
\end{thm}
\par

\begin{thm}\label{Thm:KernelsEasyDirection}
Let $K\in \maclH _0'(\rr {d_2}\times \rr {d_1})$, $s\in \overline{\mathbf R_\flat}$
and let $T$ be the linear and continuous map from $\maclH _0(\rr {d_1})$ to
$\maclH _0'(\rr {d_2})$, given by \eqref{Kmap}. Then the following is true:
\begin{enumerate}
\item if $K\in \maclH _s(\rr {d_2}\times \rr {d_1})$, then $T$ extends uniquely
to linear and continuous mappings from $\maclH _s'(\rr {d_1})$ to
$\maclH _s(\rr {d_2})$;

\vrum

\item if $K\in \maclH _s'(\rr {d_2}\times \rr {d_1})$, then $T$ extends uniquely
to linear and continuous mappings from $\maclH _s(\rr {d_1})$ to
$\maclH _s'(\rr {d_2})$.
\end{enumerate}

\par

The same holds true if the $\maclH _s$ and $\maclH _s'$ spaces are replaced
by $\maclH _{0,s}$ and $\maclH _{0,s}'$ spaces, respectively, or by
$\mascS$ and $\mascS '$ spaces, respectively.
\end{thm}

\par

\begin{proof}[Proof of Theorems \ref{Thm:KernelsDifficultDirection}
and \ref{Thm:KernelsEasyDirection}]
Let $p\in [1,\infty ]$,
$$
\vartheta _r (\alpha )=
\begin{cases}
e^{r|\alpha |^{\frac 1{2s}}},\quad &s\in \mathbf R_+{\textstyle{\bigcup}} \{ 0\} , 
\\[1ex]
r^{|\alpha |}(\alpha !)^{\frac 1{2\sigma}},\quad &s=\flat _\sigma ,
\end{cases}
$$
and $\sigma _r(\alpha )= \eabs \alpha ^r$. The results follow
from Proposition \ref{Prop:HermiteKernelThm}, and the facts that
\begin{alignat*}{3}
\maclH _s  &= \underset{r>0}{\textstyle{\bigcup}} \maclH ^p_{[\vartheta _r]},
& \quad
\maclH _{0,s}' &= \underset{r>0}{\textstyle{\bigcup}} \maclH ^p_{[1/\vartheta _r]},
& \quad
\mascS ' &= \underset{r>0}{\textstyle{\bigcup}} \maclH ^p_{[1/\sigma _r]}
\intertext{with suitable inductive limit topologies, and}
\maclH _s'  &= \underset{r>0}{\textstyle{\bigcap}} \maclH ^p_{[1/\vartheta _r]},
& \quad
\maclH _{0,s} &= \underset{r>0}{\textstyle{\bigcap}} \maclH ^p_{[\vartheta _r]},
& \quad
\mascS  &= \underset{r>0}{\textstyle{\bigcap}} \maclH ^p_{[\sigma _r]},
\end{alignat*}
with suitable inductive and projective limit topologies.
\end{proof}

\par

Evidently, the assertions on $\mascS$ and $\mascS '$ in
Theorems \ref{Thm:KernelsDifficultDirection} and
\ref{Thm:KernelsEasyDirection} are well-known. For the other cases,
the results are straight-forward consequences of the nuclearity of
$\maclH ^1_{[\vartheta ]}(\rr {d_2}\times \rr {d_1})$ (cf. e.{\,}g. \cite{GS} or
\cite{Tre}).

\par

For completeness we also write down some of the corresponding results in the
matrix case. The proofs follow by similar arguments
as for the proofs of Theorems \ref{Thm:KernelsEasyDirection}
and \ref{Thm:KernelsDifficultDirection}, and are left for the reader.
Here
%$\ell _s(\nn {d_2}\times \nn {d_1})$ ($\ell _{0,s}(\nn {d_2}\times
%\nn {d_1})$) is the set of all matrices
%%%
%\begin{align*}
%A &= (a_{\alpha ,\beta})_{(\alpha ,\beta )\in \nn {d_2}\times \nn {d_1}}
%\intertext{such that}
%|a_{\alpha ,\beta}| &\lesssim e^{-r(|\alpha |^{\frac 1{2s}} +|\beta |^{\frac 1{2s}})}
%\intertext{for some
%(for every) $r>0$, and $\ell _s'(\nn {d_2}\times \nn {d_1})$ ($\ell
%_{0,s}'(\nn {d_2}\times \nn {d_1})$) is the set of all such matrices
%such that}
%|a_{\alpha ,\beta}| &\lesssim e^{r(|\alpha |^{\frac 1{2s}} +|\beta |^{\frac 1{2s}})}
%\intertext{for every (for some) $r>0$. We also
we recall that $\ell
_{\mascS}(\nn {d_2}\times \nn {d_1})$) is the set of all matrices $A=(a_{\alpha ,\beta })
_{(\alpha ,\beta )\in \nn {d_2+d_1}}$ such that
\begin{align*}
|a_{\alpha ,\beta}| &\lesssim \eabs {(\alpha ,\beta )}^{-N}
\quad \text{for every}\quad
N\ge 0
\intertext{and $\ell _{\mascS}'(\nn {d_2}\times \nn {d_1})$) is
the set of all such matrices such that}
|a_{\alpha ,\beta}| &\lesssim \eabs {(\alpha ,\beta )}^{N}
\quad \text{for some}\quad
N\ge 0.
\end{align*}

\par

\begin{thm}\label{Thm:KernelsEasyDirection2}
Let $s\ge 0$ be real and let $T$ be the linear and continuous
map from $\ell _0(\nn {d_1})$ to $\ell _0'(\nn {d_2})$ with matrix
$A\in \ell _0'(\nn {d_2}\times \nn {d_1})$. Then the following
is true:
\begin{enumerate}
\item if $A\in \ell _s(\nn {d_2}\times \nn {d_1})$, then $T$
extends uniquely to linear and continuous mappings from
$\ell _s'(\nn {d_1})$ to $\ell _s(\nn {d_2})$;

\vrum

\item if $A\in \ell _s'(\nn {d_2}\times \nn {d_1})$, then $T$
extends uniquely to linear and continuous mappings from
$\ell _s(\nn {d_1})$ to $\ell _s'(\nn {d_2})$.
\end{enumerate}

\par

The same holds true if $\ell _s$, $\ell _s$ and their duals are replaced
by $\ell _{0,s}$, $\ell _{0,s}$ and their duals, respectively, or by
$\ell _{\mascS}$ and $\ell _{\mascS}$ and their duals, respectively.
\end{thm}

\par

\begin{thm}\label{Thm:KernelsDifficultDirection2}
Let $s\ge 0$ be real and let $T$ be the linear and continuous
map from $\ell _0(\nn {d_1})$ to $\ell _0'(\nn {d_2})$ with matrix
$A\in \ell _0'(\nn {d_2}\times \nn {d_1})$. Then the following
is true:
\begin{enumerate}
\item if $T$ is a linear and continuous map from $\ell _s'(\nn {d_1})$ to
$\ell _s(\nn {d_2})$, then $A\in \ell _s(\nn {d_2}\times \nn {d_1})$;

\vrum

\item if $T$ is a linear and continuous map from $\ell _s(\nn {d_1})$ to
$\ell _s'(\nn {d_2})$, then $A\in \ell _s'(\nn {d_2}\times \nn {d_1})$.
\end{enumerate}

\par

The same holds true if $\ell _s$, $\ell _s$ and their duals are replaced
by $\ell _{0,s}$, $\ell _{0,s}$ and their duals, respectively, or by
$\ell _{\mascS}$ and $\ell _{\mascS}$ and their duals, respectively.
\end{thm}

\par

%%%%%%%%%%%%%%%%%%%%%%%%%%%%
\section{Factorizations of Pilipovi{\'c} and Gelfand-Shilov kernels,
and pseudo-differential operators}\label{sec3}
%%%%%%%%%%%%%%%%%%%%%%%%%%%%

\par

In this section we deduce convenient factorization properties
for operators with kernels in Pilipovi{\'c} spaces. 

\par

In what follows we use the convention that if $T_0$ is a linear and continuous
operator from $\maclH _0(\rr {d_1})$ to $\maclH _0'(\rr {d_2})$,
and $g \in \maclH _0'(\rr {d_0})$, then $T_0\otimes g$ is the linear
and continuous operator from $\maclH _0(\rr {d_1})$ to
$\maclH _0'(\rr {d_2+d_0})$, given by
$$
(T_0\otimes g)\, :\, f \mapsto (T_0f)\otimes g .
$$

\par

In the following definition we recall that an operator $T$ from $\maclH _0(\rr d)$ to
$\maclH _0'(\rr d)$ is called \emph{positive semi-definite}, if
$(Tf,f)_{L^2}\ge 0$, for every $f\in \maclH _0(\rr d)$. Then we write $T\ge 0$.

\par

\begin{defn}\label{defHermdiagform}
Let $d_2\ge d_1$ and let $T$ be a linear operator from $\maclH _0(\rr {d_1})$ to
$\maclH _0'(\rr {d_2})$. Then $T$ is said to be a \emph{Hermite diagonal operator}
if $T=T_0\otimes g$, where the Hermite functions are eigenfunctions to $T_0$,
and either $d_2=d_1$ and $g=1$, or $d_2>d_1$ and $g$ is a Hermite function.

\par

Moreover, if $T=T_0\otimes g$ is on Hermite diagonal form and $T_0$ is
positive semi-definite, then
$T$ is said to be a \emph{positive semi-definite Hermite diagonal operator}.
\end{defn}

\par

The first part of the following result can be found in \cite{Beals, Vo}
(see also \cite{Kl, Sad} and the references therein for an elementary proof).

\par

\begin{thm}\label{GSkernels}
Let $s\in \mathbf R$, $T$ be a linear and continuous operator from $\maclH _0(\rr {d_1})$
to $\maclH _0'(\rr {d_2})$ with the kernel $K$, and let $d_0\ge \min (d_1,d_2)$.
Then the following is true:
\begin{enumerate}
\item If $s\ge 0$ and $K\in \maclH _{s}(\rr {d_2+d_1})$, then there
are operators $T_1$ and $T_2$ with
kernels $K_1\in \maclH _{s}(\rr {d_0+d_1})$ and $K_2\in \maclH _{s}(\rr {d_2+d_0})$
respectively such that $T=T_2\circ T_1$. 
Furthermore, if $j \in \{ 1, 2\}$ is fixed and $d _0 \geq d _j$,
then $T _j$ can be chosen as a positive semi-definte Hermite diagonal operator. 
%at least one
%of $T_1$ and $T_2$ can be chosen as positive Hermite diagonal operator;

\vrum

\item If $s>0$ and $K\in \maclH _{0,s}(\rr {d_2+d_1})$, then there
are operators $T_1$ and $T_2$ with
kernels $K_1\in \maclH _{0,s}(\rr {d_0+d_1})$ and $K_2\in \maclH
_{0,s}(\rr {d_2+d_0})$ respectively such that $T=T_2\circ T_1$.  Furthermore,
if $j \in \{ 1, 2\}$ is fixed and $d _0 \geq d _j$,
then $T _j$ can be chosen as a positive semi-definte Hermite diagonal operator. 
%at least one of $T_1$ and $T_2$ can be chosen as positive Hermite diagonal operator.
\end{enumerate}
\end{thm}

\par

The corresponding result for $s=\flat _\sigma$ reads:

\par

\begin{thm}\label{FLkernels}
Let $\sigma >0$, $T$ be a linear and continuous 
operator from $\maclH _0(\rr {d _1})$
to $\maclH '_0(\rr {d _2})$ with the 
kernel $K$. %Aso let $d _0 \geq \min (d _1, d _2)$.
Then the following is true.
\begin{enumerate}
\item If $K \in \maclH _{\flat _\sigma}(\rr {d _2 + d _1})$,
then there are operators $T _0$, $T _1$ and $T _2$ with kernels
$K _0 \in \maclH _{1/2} (\rr {d _2 + d _1})$,
$K_1 \in \maclH _{\flat _{2\sigma }}(\rr {d _1 + d _1})$
and $K_2 \in \maclH _{\flat _{2\sigma }}(\rr {d _2 + d _2})$, respectively,
and \mbox{$T = T _2 \circ T _0 \circ T _1$}. Furthermore, 
$T _1$ and $T _2$ can be chosen 
as positive semi-definite Hermite diagonal operators;

\vrum

\item If $K \in \maclH _{0, \flat _\sigma}(\rr {d _2 + d _1})$,
then there are operators $T _0$, $T _1$ and $T _2$ with kernels
$K _0 \in \maclH _{0, 1/2} (\rr {d _2 + d _1})$,
$K_1 \in \maclH _{0, \flat _{2\sigma }}(\rr {d _1 + d _1})$
and $K_2 \in \maclH _{0, \flat _{2\sigma }}(\rr {d _2 + d _2})$, respectively,
and \mbox{$T = T _2 \circ T _0 \circ T _1$}. Furthermore,
$T _1$ and $T _2$ can be chosen as positive semi-definite
Hermite diagonal operators.
\end{enumerate}
\end{thm}

\par

\begin{rem}
An operator with kernel in $\maclH _{s}(\rr {2d})$ is sometimes called a regularizing
operator with respect to $\maclH _{s}(\rr d)$, because it extends uniquely to a continuous
map from (the large space) $\maclH _{s}'(\rr d)$ into (the small space) $\maclH _{s}
(\rr d)$. Analogously, an operator with kernel in $\maclH _{0,s}(\rr {2d})$
($\mathscr S (\rr {2d})$) is sometimes called a regularizing operator with
respect to $\maclH _{0,s}(\rr d)$ ($\mathscr S(\rr d)$).
\end{rem}

\par

\begin{proof}[Proof of Theorem \ref{GSkernels}]
%We only prove (2) and (3). The assertion (1) follows by similar arguments as in the proof of (3).
%Furthermore, a proof of the first part of (1) can be found in \cite{Beals,Sad}.
%
%\par
%
First we assume that $d_0= d_1$, and start to prove (1). Let $h_{d,\alpha}(x)$
be the Hermite function on $\rr d$ of order
$\alpha \in \mathbf N^d$. Then $K$ posses the expansion
\begin{equation}\label{Kexp}
K(x,y) = \sum _{\alpha \in \mathbf N^{d_2}}\sum _{\beta \in \mathbf N^{d_1}}
a_{\alpha ,\beta}h_{d_2,\alpha}(x)h_{d_1,\beta}(y),
\end{equation}
where the coefficients $a_{\alpha ,\beta}$ satisfies
\begin{equation}\label{coeffsest}
\sup _{\alpha ,\beta} |a_{\alpha ,\beta}e^{r(|\alpha |^{\frac 1{2s}}+|\beta |^{\frac 1{2s}})}| <\infty ,
\end{equation}
for some $r>0$.

\par

Now let $z\in \rr {d_1}$, and
\begin{equation}\label{K1K2def}
\begin{aligned}
K_{0,2}(x,z) &= \sum _{\alpha \in \mathbf N^{d_2}}\sum _{\beta \in \mathbf N^{d_1}}
b_{\alpha ,\beta}h_{d_2,\alpha}(x)h_{d_1,\beta}(z),
\\[1ex]
K_{0,1}(z,y) &= \sum _{\alpha ,\beta \in \mathbf N^{d_1}}
c_{\alpha ,\beta }h_{d_1,\alpha}(z)h_{d_1,\beta}(y),
\end{aligned}
\end{equation}
where
$$
b_{\alpha ,\beta} = a_{\alpha ,\beta} e^{\frac r2|\beta|^{\frac 1{2s}}}\quad \text{and}\quad
c_{\alpha ,\beta} = \delta _{\alpha ,\beta}e^{-\frac r2|\alpha |^{\frac 1{2s}}}.
$$
Here $\delta _{\alpha ,\beta}$ is the Kronecker delta. Then it follows that
$$
\int K_{0,2}(x,z)K_{0,1}(z,y)\, dz = \sum _{\alpha \in \mathbf N^{d_2}}\sum _{\beta
\in \mathbf N^{d_1}} a_{\alpha ,\beta}h_{d_2,\alpha}(x)h_{d_1,\beta}(y)
= K(x,y).
$$
Hence, if $T_j$ is the operator with kernel $K_{0,j}$, $j=1,2$, then $T=T_2\circ T_1$.
Furthermore,
$$
\sup _{\alpha ,\beta} |b_{\alpha ,\beta}e^{\frac r2(|\alpha |^{\frac 1{2s}}+|\beta |^{\frac 1{2s}})}|
\le \sup _{\alpha ,\beta} |a_{\alpha ,\beta}e^{r(|\alpha |^{\frac 1{2s}}+|\beta |^{\frac 1{2s}})}|
<\infty 
$$
and
$$
\sup _{\alpha ,\beta} |c_{\alpha ,\beta}e^{\frac r4(|\alpha |^{\frac 1{2s}}+|\beta |^{\frac 1{2s}})}|
= \sup _{\alpha } |e^{-\frac r4 |\alpha |^{\frac 1{2s}}}e^{\frac r4 |\alpha |^{\frac 1{2s}}}|
<\infty .
$$
This implies that $K_{0,1}\in \maclH _{s}(\rr {d_1+d_1})$ and
$K_{0,2}\in \maclH _{s}(\rr {d_2+d_1})$,
and (1) follows with $K_1=K_{0,1}$ and $K_2=K_{0,2}$, in the case
$d_0=d_1$.

\par

In order to prove (2), we assume that $K\in \maclH _{0,s}(\rr {d_2+d_1})$, and we
let $a_{\alpha ,\beta}$ be the same as the above. Then
\eqref{coeffsest} holds for any $r>0$, which implies that if $n\ge 0$ is an
integer, then
\begin{equation}\label{ThetaNset}
\Theta _n \equiv \sup \sets {|\beta |}{|a_{\alpha ,\beta}|\ge
e^{-2(n+1)(|\alpha |^{\frac 1{2s}}+|\beta |^{\frac 1{2s}})}\ \text{for some}\ \alpha
\in \mathbf N^{d_2}}
\end{equation}
is finite.

\par

We let
\begin{align*}
I_1 &= \sets {\beta \in \mathbf N^{d_1}}{|\beta |\le \Theta _1+1}
\intertext{and define inductively}
I_j &= \sets {\beta \in \mathbf N^{d_1}\setminus I_{j-1}}{|\beta |\le \Theta _j+j},
\quad j\ge 2.
\end{align*}
Then
$$
I_j\cap I_k=\emptyset \quad \text{when}\quad j\neq k,\quad \text{and}\quad
\bigcup _{j\ge 0} I_j=\mathbf N^{d_1},
$$
and by the definitions it follows that $I_j$ is a finite set for every $j$.

\par

We also let $K_{0,1}$ and $K_{0,2}$ be given by \eqref{K1K2def}, where
$$
b_{\alpha _2,\beta} = a_{\alpha _2,\beta}e^{j|\beta |^{\frac 1{2s}}}\quad \text{and}\quad
c_{\alpha _1,\beta} = \delta _{\alpha _1,\beta}e^{-j|\beta |^{\frac 1{2s}}},
$$
when $\alpha _1\in \mathbf N^{d_1}$, $\alpha _2\in \mathbf N^{d_2}$
and $\beta \in I_j$. If $T_j$ is the operator with kernel $K_{0,j}$ for $j=1,2$, then it follows
that $T_2\circ T_1 =T$. Furthermore, if $r>0$, then we have
$$
\sup _{\alpha ,\beta } |b_{\alpha ,\beta} e^{r (|\alpha |^{\frac 1{2s}}+|\beta |^{\frac 1{2s}})}|
\le J_1+J_2,
$$
where
\begin{align}
J_1 &= \sup _{j\le r+1}\sup _\alpha \sup _{\beta \in I_j}|b_{\alpha ,\beta}
e^{r (|\alpha |^{\frac 1{2s}}+|\beta |^{\frac 1{2s}})}|\label{J1def}
\intertext{and}
J_2 &= \sup _{j> r+1}\sup _\alpha \sup _{\beta \in I_j}|b_{\alpha ,\beta}
e^{r (|\alpha |^{\frac 1{2s}}+|\beta |^{\frac 1{2s}})}|\label{J2def}
\end{align}

\par

Since only finite numbers of $\beta$ is involved in the suprema in \eqref{J1def},
it follows from \eqref{coeffsest} and the definition of $b_{\alpha ,\beta}$ that $J_1$
is finite.

\par

For $J_2$ we have
\begin{multline*}
J_2 = \sup _{j> r+1}\sup _\alpha \sup _{\beta \in I_j}|a_{\alpha ,\beta}
e^{r|\alpha |^{\frac 1{2s}}+(r+j)|\beta |^{\frac 1{2s}})}|
\\[1ex]
\le \sup _{j> r+1}\sup _\alpha \sup _{\beta \in I_j}|e^{-2j(|\alpha |^{\frac 1{2s}}+|\beta|^{\frac 1{2s}})}
e^{r|\alpha |^{\frac 1{2s}}+(r+j)|\beta |^{\frac 1{2s}})}| <\infty ,
\end{multline*}
where the first inequality follows from \eqref{ThetaNset}. Hence
$$
\sup _{\alpha ,\beta } |b_{\alpha ,\beta} e^{r (|\alpha |^{\frac 1{2s}}+|\beta |^{\frac 1{2s}})}|< \infty 
$$
for every $r>0$, which implies that $K_{0,2}\in \maclH _{0,s}(\rr {d_2+d_1})$.

\par

If we now replace $b_{\alpha ,\beta}$ with $c_{\alpha ,\beta}$ in the definition
of $J_1$ and $J_2$, it follows by similar arguments that both $J_1$ and $J_2$
are finite, also in this case. This gives
$$
\sup _{\alpha ,\beta } |c_{\alpha ,\beta} e^{r (|\alpha |^{\frac 1{2s}}+|\beta |^{\frac 1{2s}})}|< \infty 
$$
for every $r>0$. Hence $K_1\in \maclH _{0,s}(\rr {d_1+d_1})$, and (2) follows in the case $d_0=d_1$.

\par

Next assume that $d_0>d_1$, and let $d=d_0-d_1\ge 1$. Then we set
$$
K_1(z_0,y) = K_{0,1}(z_1,y)h_{d,0}(z)\quad \text{and}\quad
K_2(x,z_0) = K_{0,2}(x,z_1)h_{d,0}(z),
$$
where $K_{0,j}$ are the same as in the first part of the proofs, $z_1\in \rr {d_1}$
and $z\in \rr d$, giving that $z_0=(z_1,z)\in \rr {d_0}$. We get
$$
\int _{\rr {d_0}}K_2(x,z_0)K_1(z_0,y)\, dz_0 = \int _{\rr {d_1}}K_{0,2}(x,z_1)K_{0,1}(z_1,y)\,
dz_1 = K(x,y).
$$
The assertion (1) now follows in the case $d_0>d_1$ from the equivalences
\begin{alignat*}{3}
K_1 &\in \maclH _{s}(\rr {d_0+d_1})&\quad &\Longleftrightarrow & \quad
K_{0,1} &\in \maclH _{s}(\rr {d_1+d_1})
\intertext{and}
K_2 &\in \maclH _{s}(\rr {d_2+d_0})&\quad &\Longleftrightarrow & \quad
K_{0,2} &\in \maclH _{s}(\rr {d_2+d_1}),
\end{alignat*}
Since the same equivalences hold after the $\maclH _{s}$ spaces have been replaced
by $\maclH _{0,s}$ spaces, the assertion (2) also follows in the case $d_0>d_1$, and the
theorem follows in the case $d_0\ge d_1$.

\par

It remains to prove the result in the case $d_0\ge d_2$. 
By taking the adjoint, the rules of $j=1$ and
$j=2$ are interchanged, and the result follows when $d_0 \geq d_2$ as well. 
%The rules of $d_1$ and
%$d_2$ are interchanged when taking the adjoints. Hence, the result follows
%from the first part of the proof and the fact that $(x,y)\mapsto F(x,y)$ belongs to
%$\maclH _s(\rr {d_1}\times \rr {d_2})$, if and only if $(y,x)\mapsto \overline{F(x,y)}$ belongs to
%$\maclH _s(\rr {d_2}\times \rr {d_1})$.
% in combination with the facts that
%$\maclH _{s}$ and $\maclH _{0,s}$ are invariant under pullbacks of bijective linear
%transformations.
The proof is complete.
\end{proof}

\par

\begin{proof}[Proof of Theorem \ref{FLkernels}]
(1) 
We have
\begin{equation}\label{Eq:KDefAgain}
K(x, y) = \sum _{\alpha \in \bold N ^{d_2}} \sum _{\beta \in \nn {d_1}}
a _{\alpha, \beta} h_{d _2, \alpha}(x) h _{d _1, \beta}(y),
\end{equation}
where 
$$
\sup _{\alpha, \beta}|a _{\alpha, \beta}(\alpha! \beta !) ^{\frac 1{2\sigma}} R ^{-(|\alpha | + |\beta|)}|
< \infty,
$$
for some $R > 1$.

\par

Let $z _j\in \rr {d _j}$, and
\begin{align}
K _0(z_2, z_1) &= \sum _{\alpha \in \bold N ^{d_2}} \sum _{\beta \in \nn {d_1}}
a _{0, \alpha, \beta} h_{d _2, \alpha}(z_2) h _{d _1, \beta}(z_1),\label{Eq:K0Def}
\\[1ex]
K _1(z_1, y) &= \sum _{\alpha \in \bold N ^{d_1}} \sum _{\beta \in \nn {d _1}}
a _{1, \alpha, \beta} h_{d _1, \alpha}(z_1) h _{d_1, \beta}(y)\label{Eq:K1Def}
\intertext{and}
K _2(x,z_2) &= \sum _{\alpha \in \bold N ^{d _2}} \sum _{\beta \in \nn {d _2}}
a _{2, \alpha, \beta} h_{d _2, \alpha}(x) h _{d_2, \beta}(z_2),\label{Eq:K2Def}
\end{align}
where
\begin{alignat*}{2}
a _{j, \alpha, \beta} 
&= (\alpha !) ^{-\frac 1{2\sigma}} \delta _{\alpha, \beta} R^{2|\alpha|},& 
\quad \alpha ,\beta &\in \nn {d_j},
\quad j=1, 2,
\intertext{and}
a _{0, \alpha, \beta} 
&= a _{\alpha, \beta} (\alpha ! \beta !) ^{\frac 1{2\sigma}} R^ {-2(|\alpha | + |\beta |)}, &
\quad \alpha &\in \nn {d_2},\ \beta \in \nn {d_1}.
\end{alignat*}
Then it follows that
$$
\iint _{\rr {d_2+d _1}}K _2(x, z_2) K _0(z_2, z_1) K _1(z_1, y)  \, d z_2 d z_1
= K(x, y).
$$
Hence, if $T _j$ is the operator with kernel $K _j$, $j =0, 1, 2$, 
then $T = T _2 \circ T _0 \circ T _1$.
Furthermore, 
$$
\sup _{\alpha, \beta}|a _{j, \alpha, \beta}(\alpha! \beta !) ^{\frac 1{4\sigma}} R ^{-(|\alpha | + |\beta |)}|
< \infty, \quad \alpha ,\beta \in \nn {d_j} \quad j = 1, 2,
$$
and if $0<c<\log R$, then
$$
\sup _{\alpha, \beta}|a _{0, \alpha, \beta} e^{c(|\alpha | + |\beta |)}|
\leq \sup _{\alpha, \beta}|R ^{-(|\alpha | + |\beta|)} e^{c(|\alpha | + |\beta |)}| 
< \infty ,
$$
and (1) follows, in view of ??.

\par

Next we prove (2).
Let $a_{\alpha ,\beta}$ be as in \eqref{Eq:KDefAgain}. Then
$$
\sup _{\alpha, \beta}|a _{\alpha, \beta}(\alpha! \beta !) ^{\frac 1{2\sigma}} R ^{|\alpha |+ |\beta|}|
< \infty,
$$
for every $R > 1$, which implies that
\begin{align*}
\Theta _{1,n} &\equiv \sup \sets {|\beta |}{|a_{\alpha ,\beta}|\ge
(n+1) ^{-6(|\alpha |+ |\beta |)} (\alpha ! \beta !) ^{-\frac 1{2\sigma}}\ 
\text{for some}\ \alpha \in \nn {d _2}}
\intertext{and}
\Theta _{2,n} &\equiv \sup \sets {|\alpha |}{|a_{\alpha ,\beta}|\ge
(n+1) ^{-6(|\alpha |+ |\beta |)} (\alpha ! \beta !) ^{-\frac 1{2\sigma}}\ 
\text{for some}\ \beta \in \nn {d _1}}
\end{align*}
are finite for every $n\ge 1$ and $j=1,2$.

\par

We let
\begin{align*}
I_{j,1} &= \sets {\gamma \in \nn {d _j}}
{|\gamma |\le \Theta _{j,1}+1},
\intertext{and define inductively}
I_{j,m} &= \sets {\gamma \in \nn {d _j}\setminus I_{j,m-1}}
{|\gamma |\le \Theta _{j,m}+m},\quad m\ge 2,\ j=1,2.
\end{align*}
Then
$$
I_{j,m}\cap I_{j,n}=\emptyset \quad \text{when}\quad m\neq n,
\quad \text{and}\quad \underset{m\ge 1}{\textstyle{\bigcup}}  I_{j,m}=\nn {d _j}.
$$
and by the definitions it follows that $I_{j,m}$ is a finite set for every $m$.

\par

We also let $K_j$, $j = 0,1, 2$ be given by \eqref{Eq:K0Def}--\eqref{Eq:K2Def}, where
\begin{alignat*}{2}
a _{j, \alpha, \beta} 
&= (\alpha !) ^{-\frac 1{2\sigma}} \delta _{\alpha, \beta} m ^{-|\alpha+\beta|},&
\quad \alpha &\in I_{j,m}, \quad j=1, 2,
\intertext{and}
a _{0, \alpha, \beta} 
&= a _{\alpha, \beta} (\alpha ! \beta !) ^{\frac 1{2\sigma}}  m_2^{2|\alpha |}m_1 ^{2|\beta |},&
\quad \alpha &\in I_{2,m_2},\ \beta \in I_{1,m_1}.
\end{alignat*}
If $T _j$ is the operator with kernel $K _j$ for $j = 0, 1, 2$,
then it follows that $T _2 \circ T _0 \circ T _1 = T$. The result therefore follows if we
prove
\begin{equation}\label{Eq:FinalEsts}
\begin{aligned}
|a_{j,\alpha ,\beta}| &\lesssim h^{|\alpha +\beta|}(\alpha !\beta !)^{-\frac 1{4\sigma}},\ \forall h>0,\ j=1,2,
\quad \text{and}\quad
\\[1ex]
|a_{0,\alpha ,\beta}| &\lesssim e^{-c(|\alpha |+|\beta |)},\ \forall c>0,
\end{aligned}
\end{equation}
and since
$$
\underset {m\le R+1}{\textstyle{\bigcup}}I_{j,m}
\quad \text{and}\quad
\underset {m_1+m_2\le 2R}{\textstyle{\bigcup}}I_{2,m_2}\times I_{1,m_1}
$$
are finite sets and $R>1$ is arbitrary, it suffices to prove
\begin{align}
\sup _{m> R+1}\sup _{\alpha \in I _{j,m}} 
| (\alpha ! \beta !) ^{\frac 1{4\sigma}}R ^{|\alpha + \beta|} a_{j, \alpha ,\beta}| &<\infty ,
\quad j=1,2,\label{Eq:FinalEsts2A}
\intertext{and}
\sup _{(m_1,m_2)\in M_k}\sup _{\alpha \in I _{2,m_2}}
\sup _{\beta \in I _{1,m_1}} | R ^{|\alpha | + |\beta|} a_{0, \alpha ,\beta}| &<\infty ,\quad k=1,2,3,
\label{Eq:FinalEsts2B}
\end{align}
where
\begin{align*}
M_1 &= \sets {(m_1,m_2)\in \mathbf Z^2_+}{m_1\ge 2R-1,\ m_2=1}
\\[1ex]
M_2 &= \sets {(m_1,m_2)\in \mathbf Z^2_+}{m_2\ge 2R-1,\ m_1=1}
\intertext{and}
M_3 &= \sets {(m_1,m_2)\in \mathbf Z^2_+}{m_1+m_2\ge 2R,\ m_1,m_2\ge 2}
\end{align*}

\par

We have
\begin{multline*}
\sup _{m> R+1}\sup _{\alpha \in I _{j,m}} 
| (\alpha ! \beta !) ^{\frac 1{4\sigma}}R ^{|\alpha + \beta|} a_{j, \alpha ,\beta}|
\\[1ex]
=  \sup _{m> R+1}\sup _{\alpha , \beta \in I _{j,m}} 
|\delta _{\alpha, \beta}R ^{|\alpha + \beta|} m ^{-|\alpha + \beta|} |
< \infty, \quad j = 1, 2,
\end{multline*}
and \eqref{Eq:FinalEsts2A} follows.

\par

Next we prove \eqref{Eq:FinalEsts2B}, and start with the case $k=1$.
By definitions we have
$$
\alpha \in I_{2,1},\quad \beta \in I_{1,m_1},\quad m_1>2R-1>R>1
\quad \text{and}\quad
m_2=1.
$$
Then
$$
|a_{\alpha ,\beta}|\le m_1^{-6(|\alpha |+|\beta |)}(\alpha !\beta !)^{-\frac 1{2\sigma}},
$$
which implies
$$
|a_{0,\alpha ,\beta}|=|a_{\alpha ,\beta}|(\alpha !\beta !)^{\frac 1{2\sigma}}m_1^{2|\beta |}
\le m_1^{-4(|\alpha |+|\beta |)}.
$$
Hence,
$$
\sup _{(m_1,m_2)\in M_1}\sup _{\alpha \in I _{2,1}}
\sup _{\beta \in I _{1,m_1}} | R ^{|\alpha | + |\beta|} a_{0, \alpha ,\beta}|
\le
\sup _{m_1>R>1}R ^{|\alpha | + |\beta|} m_1^{-4(|\alpha | + |\beta|)} <\infty ,
$$
and \eqref{Eq:FinalEsts2B} follows in the case $k=1$.

\par

In the same way, \eqref{Eq:FinalEsts2B} follows in the case $k=2$.

\par

Next we prove \eqref{Eq:FinalEsts2B} in the case $k=3$. By the definitions
it follows that if $\alpha \in I_{2,m_2}$, then
\begin{align*}
|a_{\alpha ,\beta}| &\le m_2^{-6(|\alpha |+|\beta |)}(\alpha !\beta !)^{-\frac 1{2\sigma}},
\quad \forall \ \beta \in \nn {d_1},
\intertext{and if $\beta \in I_{1,m_1}$, then}
|a_{\alpha ,\beta}| &\le m_1^{-6(|\alpha |+|\beta |)}(\alpha !\beta !)^{-\frac 1{2\sigma}},
\quad \forall \ \alpha \in \nn {d_2}.
\intertext{Hence, if $\alpha \in I_{2,m_2}$ and $\beta \in I_{1,m_1}$, the geometric mean of
the latter inequalities becomes}
|a_{\alpha ,\beta }| &\le (m_1m_2)^{-3(|\alpha |+|\beta |)}(\alpha !\beta !)^{-\frac 1{2\sigma}},
\intertext{giving that}
|a_{0,\alpha ,\beta }| &\le (m_1m_2)^{-(|\alpha |+|\beta |)}.
\end{align*}

\par

This gives
\begin{multline*}
\sup _{(m_1,m_2)\in M_3}\sup _{\alpha \in I _{2,m_2}}\sup _{\beta \in I _{1,m_1}} 
| R ^{|\alpha | + |\beta|} a_{0, \alpha ,\beta}|
\\[1ex]
=  \sup _{(m_1,m_2)\in M_3}\sup _{\alpha \in I _{2,m_2}}\sup _{\beta \in I _{1,m_1}} 
R ^{|\alpha | + |\beta|} (R(m_1m_2)^{-1})^{|\alpha |+|\beta |}
< \infty ,
\end{multline*}
and \eqref{Eq:FinalEsts2B}, and thereby \eqref{Eq:FinalEsts} follow.

\par

By \eqref{Eq:FinalEsts} it follows that
$K_0 \in \maclH _{0, 1/2}(\rr {d _2 + d _1})$,
$K _1 \in \maclH _{0, \flat _{2\sigma }} (\rr {d _1 + d _1})$ and
$K_2 \in \maclH _{0, \flat _{2\sigma }}(\rr {d _2 + d _2})$. Furthermore,
$T_1$ and $T_2$ are positive semi-definite Hermite diagonal operators by construction.
\end{proof}

\par

\begin{rem}\label{Rem:DiagFactorizations}
Let $\sigma >0$ and $T\ge 0$ be a Hermite diagonal operator on $L^2(\rr d)$
with kernel $K$ in $\maclH _{\flat _\sigma}$.
By the proof of Theorem \ref{FLkernels},
Then there are
Hermite diagonal operators $T_1\ge 0$ and $T_2\ge 0$ on $L^2(\rr d)$ with
kernels $K_1$ and $K_2$ such that
%satisfying
$$
K_1\in \maclH _{\flat _\sigma}(\rr {2d}),\quad K_2\in \maclH _{1/2}(\rr {2d})
\quad \text{and}\quad T=T_1\circ T_2=T_2\circ T_1.
$$

\par

In fact, if $K$ is given by \eqref{Eq:KDefAgain} with $d_1=d_2=d$, it suffices
to let $K_1$ and $K_2$ be given by \eqref{Eq:K1Def} and \eqref{Eq:K2Def},
with
$$
a _{1, \alpha, \beta} = (a _{\alpha, \beta})^{1/2} (\alpha !\beta !)^{-\frac 1{4\sigma}}
\quad \text{and}\quad
a _{2, \alpha, \beta} = (a _{\alpha, \beta})^{1/2}(\alpha !\beta !)^{\frac 1{4\sigma}}.
$$
%
%
%$K_1\in \maclH _{1/2}(\rr {2d})$ and $K_2\in \maclH _{\flat _\sigma}(\rr {2d})$
%The same holds true if the $\maclH _{\flat _\sigma}$ and $\maclH _{1/2}$
%spaces are replaced by $\maclH _{0,\flat _\sigma}$ and $\maclH _{0,1/2}$, respectively.
\end{rem}

\par

\begin{rem}
From the construction of $K_1$ and $K_2$ in the proofs of Theorems \ref{GSkernels}
and \ref{FLkernels}, it follows
that it is not so complicated for using numerical methods when
obtaining approximations of candidates to $K_1$ and $K_2$. In fact, $K_1$ and $K_2$ are formed
explicitly by the elements of the matrix for $T$, when the Hermite functions are used as
basis for $\mathscr S$, $\maclS _s$ and $\Sigma _s$.
\end{rem}

\par

The following result is an immediate consequence of Theorem \ref{GSkernels} and
the fact that the map $a\mapsto K_{a,t}$ is continuous and bijective on $\maclS
_{s_1}(\rr {2d})$, and on $\Sigma _{s_2}(\rr {2d})$, for every $s_1\ge 1/2$, $s_2>1/2$
and $t\in \mathbf R$.

\par

\begin{thm}\label{GSpseudo}
Let $A$ be a real $d\times d$-matrix, $s_1\ge \frac 12$ and let
$s_2>\frac 12$. Then the following is true:
\begin{enumerate}
\item if $a\in \maclH _{s_1}(\rr {2d})$, then there are $a_1,a_2\in \maclH
_{s_1}(\rr {2d})$ such that $a=a_1\wpr _Aa_2$;

\vrum

\item  if $a\in \maclH _{0,s_2}(\rr {2d})$, then there are
$a_1,a_2\in \maclH _{0,s_2}(\rr {2d})$ such that $a=a_1\wpr _Aa_2$.
\end{enumerate}
\end{thm}

\par

\begin{rem}
Extensions of Theorem \ref{GSpseudo} to the case $s_1,s_2\ge 0$ is not so smooth, because
the Pilipovi{\'c} spaces which are not Gelfand-Shilov spaces are not invariant under
dilations. However, if $A$ is a real $d\times d$ matrix and $a\in \mascS (\rr {2d})$
is such that the kernel
\begin{equation}\label{Eq:OpAKernel}
(x,y)\mapsto (\mascF _2^{-1}a)(x-A(x-y),x-y)
\end{equation}
belongs to $\maclH _s(\rr {2d})$, then we may apply Theorem \ref{GSkernels} in
this situation as well.

\par

Therefore, let $\mathcal G_{A,s}(\rr {2d})$ ($\mathcal G_{0,A,s}(\rr {2d})$)
be the set of all $a\in \mascS (\rr {2d})$ such that the map in \eqref{Eq:OpAKernel}
belongs to $\maclH _s(\rr {2d})$ ($\maclH _{0,s}(\rr {2d})$). If
$a\in \mathcal G_{A,s}(\rr {2d})$ ($a\in \mathcal G_{0,A,s}(\rr {2d})$), then
there are $a_1,a_2\in \mathcal G_{A,s}(\rr {2d})$ ($a_1,a_2\in \mathcal G_{0,A,s}(\rr {2d})$)
such that $a=a_1\wpr _Aa_2$.
\end{rem}

\par

%%%%%%%%%%%%%%%%%%%%%%%%%%%%
\section{Singular value estimates and Schatten-von Neumann
properties for operators with Gelfand-Shilov kernels}\label{sec4}
%%%%%%%%%%%%%%%%%%%%%%%%%%%%

\par

In this section we use Theorem \ref{GSkernels} to
obtain estimates of the form \eqref{Eq:StrongSVEst1} for operators $T$ with kernels
in Pilipovi{\'c} spaces of order $s$, provided $\mascB _1$ and $\mascB _2$ stay
between the given Pilipovi{\'c} space and its dual. In particular it follows that any
such operator belongs to any Schatten-von Neumann class.

\par

We start by recalling the definition of quasi-Banach spaces. Let $\mascB$ be a
vector space. A
\emph{quasi-norm} $\nm \cdo {\mascB}$ on $\mascB$ is a non-negative and
real-valued function on $\mascB$ which is non-degenerate in the sense
$$
\nm {f}{\mascB}=0\qquad \Longleftrightarrow \qquad f=0,\quad f\in \mascB,
$$
and fulfills
\begin{alignat}{2}
\nm {\alpha f}{\mascB} &= |\alpha |\cdot \nm f{\mascB}, &\qquad 
f &\in \mascB ,\ \alpha \in \mathbf C\notag
\intertext{and}
\nm {f+g}{\mascB} &\le D(\nm f{\mascB} +\nm g{\mascB}),
&\qquad f,g &\in \mascB , \label{weaktriangle}
\end{alignat}
for some constant $D\ge 1$ which is independent of $f,g\in \mascB$. Then
$\mascB$ is a topological vector space when the topology for $\mascB$
is defined by $\nm \cdo{\mascB}$, and $\mascB$ is called a quasi-Banach
space if $\mascB$ is complete under this topology.

\par

Let $\mascB _1$ and $\mascB _2$ be (quasi-)Banach spaces, and let
$T$ be a linear map between $\mascB _1$ and $\mascB _2$. Then
recall that the \emph{singular values} of order $j\ge 1$ of
$T$ is given by \eqref{Eq:SingValDef}, where
the infimum is taken over all linear operators $T_0$ from $\mascB _1$
to $\mascB _2$ with rank at most $j-1$. Therefore, $\sigma _1(T)$
equals $\nm T{\mascB _1\to \mascB _2}$, and $\sigma _k(T)$ are 
non-negative which decrease with $k$.

\par

For any $p\in (0,\infty ]$, $\mascI _p(\mascB _1,\mascB _2)$, the set of
Schatten-von Neumann operators of order $p$ is the set of all
linear and continuous operators $T$ from $\mascB _1$ to $\mascB _2$
such that
$$
\nm T{\mascI _p} = \nm T{\mascI _p(\mascB _1,\mascB _2)}
\equiv \nm { ( \sigma _k(T,\mascB _1,\mascB _2) ) _{k=1}^\infty}{l^p}
$$
is finite.

\par

In the following result we show that
the singular values for operators $T_K$ with kernels $K$ in
Pilipovi{\'c} spaces or Schwartz spaces, obey estimates of the form
\begin{align}
\sigma _k(T_K,\mascB _1,\mascB _2) &\lesssim e^{-ck^{\frac 1{2ds}}},
\label{SingValPilSpEst}
\\[1ex]
\sigma _k(T_K,\mascB _1,\mascB _2) &\lesssim R^{k} (k!)^{-\frac 1{2\sigma d}}
\label{SingValPilSpEst2}
\intertext{or}
\sigma _k(T_K,\mascB _1,\mascB _2) &\lesssim k^{-N}.
\label{SingValSchwSpEst}
\end{align}
We observe that (5) should be available in the literature.

\par

\begin{thm}\label{thmSchattGSkernels}
Let %$\mascB _1$ and $\mascB _2$ be quasi-Banach spaces such that
%$$
%\maclH _0(\rr {d_j})\hookrightarrow \mascB _j \hookrightarrow
%\maclH _0'(\rr {d_j}),\quad j=1,2,
%$$
$p\in (0,\infty ]$, $s\ge 0$, $\sigma >0$ and let
$d=\min (d_1,d_2)$. Then the following is true:
\begin{enumerate}
\item if $K\in \maclH _s(\rr {d_2+d_1})$, and
$\mascB _1$ and $\mascB _2$ are quasi-Banach spaces
such that $\mascB _1\hookrightarrow \maclH _s'(\rr {d_1})$ and
$\maclH _s(\rr {d_2})\hookrightarrow \mascB _2$, then \eqref{SingValPilSpEst}
holds for some $c>0$. In particular,
$T_K\in \mascI _p(\mascB _1,\mascB _2)$;

\vrum

\item if $K\in \maclH _{0,s}(\rr {d_2+d_1})$, and $\mascB _1$ and
$\mascB _2$ are quasi-Banach spaces
such that $\mascB _1\hookrightarrow \maclH _{0,s}'(\rr {d_1})$
and $\maclH _{0,s}(\rr {d_2})\hookrightarrow \mascB _2$, then \eqref{SingValPilSpEst} holds
for every $c>0$. In particular, $T_K\in \mascI _p(\mascB _1,\mascB _2)$;

\vrum

\item if $K\in \maclH _{\flat _\sigma}(\rr {d_2+d_1})$, and $\mascB _1$ and
$\mascB _2$ are quasi-Banach spaces
such that $\mascB _1\hookrightarrow \maclH _{1/2}'(\rr {d_1})$
and $\maclH _{1/2}(\rr {d_2})\hookrightarrow \mascB _2$, then
\eqref{SingValPilSpEst2} holds
for some $c>0$. In particular, $T_K\in \mascI _p(\mascB _1,\mascB _2)$;

\vrum

\item if $K\in \maclH _{0,\flat _\sigma}(\rr {d_2+d_1})$, and
$\mascB _1$ and $\mascB _2$ are quasi-Banach spaces
such that $\mascB _1\hookrightarrow \maclH _{0,1/2}'(\rr {d_1})$
and $\maclH _{0,1/2}(\rr {d_2})\hookrightarrow \mascB _2$,
then \eqref{SingValPilSpEst2} holds
for every $c>0$. In particular, $T_K\in \mascI _p(\mascB _1,\mascB _2)$;

\vrum

\item if $K\in \mascS (\rr {d_2+d_1})$, and $\mascB _1$ and
$\mascB _2$ are quasi-Banach spaces
such that $\mascB _1\hookrightarrow \mascS '(\rr {d_1})$ and
$\mascS (\rr {d_2})\hookrightarrow \mascB _2$,
then \eqref{SingValSchwSpEst} holds
for every $N>0$. In particular, $T_K\in \mascI _p(\mascB _1,\mascB _2)$.
\end{enumerate}
\end{thm}

\par

We need some preparations for the proof. First
we recall that if $\mascB _j$, $j=0,1,2$, are quasi-Banach
spaces and $T_j$ are linear and continuous mappings from
$\mascB _{j-1}$ to $\mascB _j$, $j=1,2$, then
\begin{align}
\sigma _k(T_2\circ T_1,\mascB _0,\mascB _2) &\le C \nm {T_1}{\mascB _0\to \mascB _1}
\sigma _k(T_2,\mascB _1,\mascB _2)\label{EqSingValEst1}
\intertext{and}
\sigma _k(T_2\circ T_1,\mascB _0,\mascB _2) &\le C \nm {T_2}{\mascB _1\to \mascB _2}
\sigma _k(T_1,\mascB _0,\mascB _1).\label{EqSingValEst2}
\end{align}

\par

In fact, if $\Omega _{j,l}(k)$ is the set of all linear operators from $\mascB _j$
to $\mascB _l$ with rank at most $k-1$, then
\begin{multline*}
\sigma _k(T_2\circ T _1,\mascB _0,\mascB _2) = \inf _{S\in \Omega _{0,2}(k)}
\nm {T_2\circ T_1-S}{\mascB _0\to \mascB _2}
\\[1ex]
\le \inf _{T_0\in \Omega _{0,1}(k)}
\nm {T_2\circ T_1-T_2\circ T_0}{\mascB _0\to \mascB _2}
\\[1ex]
\le C \nm {T_2}{\mascB _1\to \mascB _2}\left (  \inf _{T_0\in \Omega _{0,1}(k)}
\nm {T_1-T_0}{\mascB _0\to \mascB _1} \right )
\\[1ex]
=C\nm {T_2}{\mascB _1\to \mascB _2} \sigma _k(T _1,\mascB _0,\mascB _1),
\end{multline*}
which gives \eqref{EqSingValEst2}. In the same way \eqref{EqSingValEst1}
is obtained. (See also \cite{Si}.)

\par

\begin{proof}[Proof of Theorem \ref{thmSchattGSkernels}]
We only prove (1), (3) and (5). The assertions (2) and (4) follow by
similar arguments and are left for the reader.

\par

(1) First we prove the result with $d_0=\max (d_1,d_2)$ in place of $d$.
By Theorem \ref{GSkernels} we get
\begin{equation}\label{Eq:T_KFact321}
T_K = T_{K_3}\circ T_{K_2}\circ T_{K_1},
\end{equation}
where the kernels $K_1$, $K_2$ and $K_3$ of the operators $T_{K_1}$,
$T_{K_2}$ and $T_{K_3}$ belong to $\maclH _s(\rr {{d_0}+d_1})$,
$\maclH _s(\rr {{d_0}+{d_0}})$ and $\maclH _s(\rr {d_2+{d_0}})$, respectively. Furthermore,
we may assume that $T_{K_2}$ is a positive semi-definite Hermite diagonal
operator.

\par

It follows that $T_{K_1}$ is continuous from $\mascB _1$ to $L^2(\rr {d_0})$,
and $T_{K_3}$ is continuous from $L^2(\rr {d_0})$ to $\mascB _2$. Hence, by
\eqref{EqSingValEst1} and \eqref{EqSingValEst2} it suffices to prove that if
$T=T_{K_2}$, then
\begin{equation}\label{EqRedSingEst}
\sigma _k = \sigma _k(T,L^2(\rr {d_0}),L^2(\rr {d_0})) \le Ce^{-ck^{\frac 1{2{d_0}s}}}.
\end{equation}
%%
%considered as an operator on $L^2(\rr d)$.

\par

By the constructions we have
\begin{align}
K_2(x,y) &= \sum _{\alpha \in \nn {d_0}}c_\alpha h_\alpha (x)h_\alpha (y)
\label{Eq:K2HemrExp}
\intertext{where}
0\le c_\alpha \lesssim e^{-c|\alpha |^{\frac 1{2s}}}
\label{EqDiagCoeffEst}
\end{align}
for some constant $c>0$. Furthermore, $\{ h_\alpha \} _{\alpha \in \nn {d_0}}$ is an
orthonormal basis of eigenfunctions with eigenvalues $\{ c_\alpha \} _{\alpha \in \nn {d_0}}$
of $T$.

\par

Now let $M_{N,{d_0}}$ be the number of all multi-indices $\alpha \in \nn {d_0}$
such that $|\alpha |\le N$. Then $M_{N,{d_0}}\asymp \eabs N^{d_0}$. Since
the singular values are the eigenvalues of $T_{K_2}$ in decreasing order,
\eqref{EqDiagCoeffEst} gives
$$
\sigma _k(T_{K_2}) \lesssim e^{-cN ^{\frac 1{2s}}}
$$
for some $c>0$ when $M_{N-1,{d_0}}<k\le M_{N,{d_0}}$. For such $k$
we also have $k\asymp N^{d_0}$, since $(N-1)^{d_0}\asymp N^{d_0}$.
By combining these estimates we get
$$
\sigma _k (T_K)\lesssim e^{-cN ^{\frac 1{2s}}} \lesssim e^{-c_0k ^{\frac 1{2{d_0}s}}},
$$
for some constant $c_0$. This gives \eqref{SingValPilSpEst} with $d_0$ in place of $d$.

\par

It remains to prove that we may replace $d_0$ by $d$ in our estimates.
We consider
$$
T_1 =T^*_K\circ T_K
\quad \text{and}\quad
T_2 =T_K\circ T_K^*,
$$
which are non-negative and with kernels $K_1$ and $K_2$ in
$\maclH _s(\rr {d_1+d_1})$ and $\maclH _s(\rr {d_2+d_2})$,
respectively. Hence, by the first part of the proof we get
\begin{equation}\label{Eq:SingUppsk}
\sigma _k(T_1) \lesssim e^{-ck^{\frac 1{2d_1s}}}
\quad \text{and}\quad
\sigma _k(T_2) \lesssim e^{-ck^{\frac 1{2d_2s}}}.
\end{equation}
Since
\begin{equation}\label{Eq:SingEqRel}
\sigma _k(T_1) = \sigma _k(T_2) = \sigma _k(T_K)^2,
\end{equation}
\eqref{SingValPilSpEst} now follows from \eqref{Eq:SingUppsk}
and \eqref{Eq:SingEqRel}.

\par

(3) By Theorem \ref{FLkernels} and Remark \ref{Rem:DiagFactorizations}, we get
\begin{equation}\label{Eq:TKComps5times}
T _K = T _{K _{0,2}} \circ T _{K _2} \circ T _{K _0} \circ T _{K _1} \circ T _{K _{0,1}},
\end{equation}
where the corresponding kernels satisfy
$$
K_{0,j} \in \maclH _{1/2} (\rr {d_j+d _j}), \quad K_j \in
\maclH _{\flat _{2\sigma}}(\rr {d_j+d_j}) ,\quad
\text{and}\quad K_0\in \maclH _{1/2} (\rr {d_2+d_1}),
$$
$j=1,2$.
%%
%%%
%\begin{alignat*}{2}
%K_{0,1} &\in \maclH _{1/2} (\rr {d_1+d _1}),& \quad K_1 &\in
%\maclH _{\flat _{\sigma /4}}(\rr {d_1+d_1}) ,\quad K_0\in \maclH _{1/2} (\rr {d_2+d_1}),
%\\[1ex]
%K_2 &\in \maclH _{\flat _{\sigma /4}}(\rr {d_2+d_2}), & \quad
%K_{0,2} &\in \maclH _{1/2} (\rr {d _2+d_2})
%\end{alignat*}
%%%
%
%
%the corresponding kernels $K _1$, $K _2$, 
%$K _3$, $K _4$ and $K _5$ of the operators 
%$T _{K _1}$, $T _{K _2}$, $T _{K _3}$, $T _{K _4}$
%and $T _{K _5}$ belong to $\maclH _{1/2} (\rr {d+d _1})$,
%
%$\maclH _{\flat _{1/2}}(\rr {d+d})$,
%
%$\maclH _{1/2} (\rr {d+d})$, 
%
%$\maclH _{\flat _{1/2}}(\rr {d+d})$,
%and $\maclH _{1/2} (\rr {d _2+d})$, respectively.
Furthermore, all kernels except $K_0$ to the operators in
\eqref{Eq:TKComps5times} are positive semi-definite Hermite diagonal
operators.

\par

It follows that
\begin{gather*}
T _{K _{0,1}}\, :\, \mascB _1 \to L ^2(\rr {d_1}),
\quad
T_{K_0}\, :\, L^2(\rr {d_1}) \to L^2(\rr {d_2})
\\[1ex]
\text{and}\quad
T _{K _{0,2}}\, :\, L ^2(\rr {d_2}) \to \mascB _2
\end{gather*}
are continuous. By similar arguments as in the proof of
(1), we now get
% is continuous from $\mascB_1$
%to $L ^2(\rr d)$,
%$T _{K _3}$ is continuous from $L ^2(\rr d)$ to $L ^2(\rr d)$,
%and $T _{K _5}$ is continuous from $L ^2(\rr d)$ to $\mascB_2$.
%By following the similar way as in the last part of the proof of
%Theorem 3.1, we get
$$
\sigma _k(T _{K _{j}}, L ^2({\rr {d_j}}), L ^2({\rr {d_j}}))
\lesssim R^{k} (k!)^{-\frac 1{2\sigma d_j}}, \quad j = 2, 4.
$$
Hence,
$$
\sigma _k(T _K,\mascB _1, \mascB _2)\lesssim R^{k} (k!)^{-\frac 1{2\sigma d_j}},
\quad j=1,2,
$$
in view of \eqref{EqSingValEst1}--\eqref{EqSingValEst2}. This gives (3).

%By Fan's inequality we have
%%%
%\begin{multline*}
%\sigma _{j+k+1} (T _{K_4} \circ T _{K_3} \circ T _{K_2},  L ^2({\rr d}), L ^2({\rr d})) 
%\\[1ex]
%\lesssim \sigma _{j+1} (T _{K _4}, L ^2({\rr d}), L ^2({\rr d}))
%\cdot \sigma _{k+1} (T _{K _2}, L ^2({\rr d}), L ^2({\rr d}))
%\\[1ex]
%\lesssim  R _1^{j+1} R _2^{k+1} ((j+1) !)^{-1/{2d}} ((k+1) !)^{-1/{2d}}
%\lesssim  R _3 ^{j+k+1} ((j+k+1)!)^{-1/{2d}}
%\end{multline*}
%%%

\par

(5) By \cite{Beals, Kl, Sad, Vo} we get
\begin{equation}\label{Eq:T_KFact321}
T_K = T_{K_3}\circ T_{K_2}\circ T_{K_1},
\end{equation}
where the kernels $K_1$, $K_2$ and $K_3$ of the operators $T_{K_1}$,
$T_{K_2}$ and $T_{K_3}$ belong to $\mascS (\rr {d_0+d_1})$,
$\mascS (\rr {d_0+d_0})$ and $\mascS (\rr {d_2+d_0})$, respectively. Furthermore,
we may assume that $T_{K_2}$ is a positive semi-definite Hermite diagonal
operator (cf. e.{\,}g. \cite{ToKrNiNo}).

\par

It follows that $T_{K_1}$ is continuous from $\mascB _1$ to $L^2(\rr {d_0})$,
and $T_{K_3}$ is continuous from $L^2(\rr {d_0})$ to $\mascB _2$. Hence, by
\eqref{EqSingValEst1} and \eqref{EqSingValEst2} it suffices to prove that
for every $N>0$ there is a constant $C>0$ such that
\begin{equation}\label{EqRedSingEst2}
\sigma _k = \sigma _k(T_{K_2},L^2(\rr {d_0}),L^2(\rr {d_0})) \le Ck^{-N}.
\end{equation}
%%
%considered as an operator on $L^2(\rr d)$.

\par

By the constructions we have $K_2$ is given by \eqref{Eq:K2HemrExp}, where
$c_\alpha$ fulfills
$$
0\le c_\alpha \lesssim \eabs \alpha ^{-N},
$$
for every $N>0$. By the same arguments as in the first part of the proof we now get
$$
\sigma _k \lesssim k^{-\frac N{d_0}}
$$
for every $N$, and (5) follows. %for $d_0$ in place of $d$.

\par

Finally, by \eqref{SingValPilSpEst}--\eqref{SingValSchwSpEst} it also follows that
$\{ \sigma _k(T,\mascB _1,\mascB _2)\}$ belongs to $\ell ^p$ for every $p>0$.
This gives the second parts of (1)--(5). % for $d_0$ in place of $d$.
\end{proof}

\par

%%%%%%%%%%%%%%%%%%%%%%%%%%%%%
\section{Discrete characterizations of kernels to
smoothing operators}\label{sec5}
%%%%%%%%%%%%%%%%%%%%%%%%%%%%%

\par

In this section we show that any operators with kernels in
Gelfand-Shilov, Pilipovi{\'c} or Schwartz spaces can be characterized
by convenient expansions of the form
%%%
%\begin{equation}\label{Eq:KernelExp1}
%K(x_2,x_1) = \sum _{j=1}^\infty \lambda _j f_{1,j}(x_1)f_{2,j}(x_2),
%\quad
%x_k\in \rr {d_k},\ \{f_{k,j} \} _{j=1}^\infty \in \OG (L^2(\rr {d_k})). 
%\end{equation}
%%%
%%
\begin{align}
K &= \sum _{j=1}^\infty \lambda _j f_{1,j}\otimes f_{2,j},
\label{Eq:KernelExp1}
\intertext{where}
\{\lambda _j \} _{j=1}^\infty &\subseteq \mathbf R_+,\quad
\{f_{k,j} \} _{j=1}^\infty \in \OG (L^2(\rr {d_k})),\ k=1,2. 
\label{Eq:KernelExp2}
\end{align}
Here $\OG (L^2(\rr d))$ is the set of all orthogonal sequences
$\{ f_j\} _{j=1}^\infty$ in $L^2(\rr d)$, i.{\,}e., $0\neq f_j\in L^2(\rr d)$,
and $f_{j_1}\bot f_{j_2}$when $j_1\neq j_2$. Note that we do \emph{not}
require that $f_j$ should be normalized with respect to $L^2(\rr d)$.

\par

For future references we also let $\ON (L^2(\rr d))$ be the set of all
orthonormal sequences in $L^2(\rr d)$.

\par

\begin{thm}\label{Thm:KernSchwSpaceChar}
Let $p\in [1,\infty ]$ and $T$ be a linear and continuous operator
from $\maclH _0(\rr {d_1})$ to $\maclH _0'(\rr {d_2})$ with
kernel $K\in C^\infty (\rr {d_2}\times \rr {d_1})$. Then the following is true:
\begin{enumerate}
\item if $K\in \mascS (\rr {d_2}\times \rr {d_1})$, then \eqref{Eq:KernelExp1}
holds for some sequences in \eqref{Eq:KernelExp2} such that
\begin{equation}\label{Eq:KernSchwSpaceChar}
\sup _{j\ge 1} \big (j^N\lambda _j \big ) <\infty
\quad \text{and}\quad
\sup _{j\ge 1}\left (  j^N\nm {x^\alpha D^\beta f_{k,j}}{L^p(\rr {d_k})}\right )
<\infty ,
\end{equation}
for $k=1,2$, $\alpha ,\beta \in \nn d$ and every $N\ge 0$.

\vrum

\item on the other hand, if instead $K\in C^\infty (\rr {d_2}\times \rr {d_1})$
and satisfies \eqref{Eq:KernelExp1} for some $\{ \lambda _j\}
_{j=1}^\infty \subseteq \mathbf R_+$, and
\eqref{Eq:KernSchwSpaceChar} holds for $k=1,2$ and every $N\ge 1$,
then $K\in \mascS (\rr {d_2}\times \rr {d_1})$.
\end{enumerate}
\end{thm}

\par

The corresponding characterizations of operators with Pilipovi{\'c} kernels are
given in the following theorem.

\par

\begin{thm}\label{Thm:KernPilSpaceChar}
Let $p\in [1,\infty ]$, $s>0$, $d=\min (d_1,d_2)$ and $T$ be a linear
and continuous operator from $\maclH _0(\rr {d_1})$ to
$\maclH _0'(\rr {d_2})$ with kernel $K\in C^\infty (\rr {d_2}\times \rr {d_1})$.
Then the following is true:
\begin{enumerate}
\item if $K\in \maclH _s (\rr {d_2}\times \rr {d_1})$ ($K\in \maclH _{0,s}
(\rr {d_2}\times \rr {d_1})$), then \eqref{Eq:KernelExp1}
holds for some sequences in \eqref{Eq:KernelExp2} such that
\begin{equation}\label{Eq:KernPilRoumSpaceChar}
\sup _{j\ge 1} \big (e^{r\cdot j^{\frac 1{2ds}}} \lambda _j \big ) <\infty
\quad \text{and}\quad
\sup \left ( \frac {e^{r\cdot j^{\frac 1{2ds}}}
\nm {H^N f_{k,j}}{L^p(\rr {d_k})}}{h^N(N!)^{2s}}\right )
<\infty ,
\end{equation}
for $k=1,2$ and some $h> 0$ and $r>0$ (every $h> 0$ and $r>0$),
where the latter supremum is taken over all $j\ge 0$ and $N\ge 0$;

\vrum

\item on the other hand, if instead $K\in C^\infty (\rr {d_2}\times \rr {d_1})$
and satisfies \eqref{Eq:KernelExp1} for some $\{ \lambda _j\}
_{j=1}^\infty \subseteq \mathbf R_+$, and
\eqref{Eq:KernPilRoumSpaceChar} holds for $k=1,2$ and some $r>0$
(every $r>0$),
then $K\in \maclH _s (\rr {d_2}\times \rr {d_1})$ ($K\in \maclH _{0,s}
(\rr {d_2}\times \rr {d_1})$).
\end{enumerate}
\end{thm}

\par

We need some preparations for the proof. First we observe
that $\maclH ^p_{[\vartheta ]}$ possess the expected interpolation
properties.

\par

\begin{lemma}\label{Lemma:Interpolation1}
Let $\theta \in [0,1]$, $\vartheta$, $\vartheta _1$ and $\vartheta _2$
be weights on $\nn d$, and let $p,p_1,p_2\in [1,\infty ]$ be such that
$$
\frac 1p=\frac {1-\theta}{p_1}+\frac \theta{p_2}
\quad \text{and}\quad
\vartheta = \vartheta _1^{1-\theta}\vartheta _2^\theta.
$$
Then
$$
(\maclH ^{p_1}_{[\vartheta _1]}(\rr d),\maclH ^{p_2}_{[\vartheta _2]}(\rr d))
_{[\theta ]} =\maclH ^p_{[\vartheta ]}(\rr d).
$$
\end{lemma}

\par

\begin{proof}
The result follows from the fact that the map
$$
\{ c_\alpha \} _{\alpha \in \nn d} \mapsto \sum _{\alpha \in \nn d}
c_\alpha h_\alpha
$$
is bijective and isometric from $\ell ^p_{[\vartheta ]}(\nn d)$ to
$\maclH ^p_{[\vartheta ]}(\rr d)$, and that
$$
(\ell ^{p_1}_{[\vartheta _1]}(\nn d),\ell ^{p_2}_{[\vartheta _2]}(\nn d))
_{[\theta ]} =\ell ^p_{[\vartheta ]}(\nn d).
$$
\end{proof}

\par

 We also need the following result on powers of
non-negative self-adjoint operators on $L^2(\rr d)$.

\par

\begin{prop}\label{Prop:SelfAdj}
Let $s\ge 0$, $r>0$ and let $T$ be a self-adjoint and non-negative operator
on $L^2(\rr d)$ with kernel $K$ in $\maclH _s(\rr d\times \rr d)$. Then
the following is true:
\begin{enumerate}
\item the kernel of $T^r$ belongs to $\maclH _s(\rr d\times \rr d)$;

\vrum

\item $T^r$ is continuous from $\maclH _s'(\rr d)$ to $\maclH _s(\rr d)$.
\end{enumerate}

\par

The same holds true if the $\maclH _s$ and $\maclH _s'$ spaces
are replaced by $\maclH _{0,s}$ and $\maclH _{0,s}'$ spaces, respectively,
or by $\mascS$ and $\mascS '$ spaces, respectively.
\end{prop}

\par

\begin{proof}
We only prove the result when $K\in \maclH _s(\rr d\times \rr d)$. The other
cases follows by similar arguments and is left for the reader.

\par

Let
$$
\Omega = \sets {z\in \mathbf C}{0<\operatorname{Re}(z)<1}
$$
and $T_0(z)=T^z$ when $z\in \overline \Omega$. Then the map
$z\mapsto T(z)$ with values in $\mathscr L(L^2(\rr d))$
is continuous on $\overline \Omega$ and analytic on $\Omega$.

\par

Furthermore, by writing $T^z=T^x\circ T^{iy}$ when $z=x+iy$, and
using that $T^{iy}$ is bounded on $L^2(\rr d)$ when $y\in \mathbf R$,
it follows from the assumptions that
\begin{align*}
\sup _{y\in \mathbf R}\nm {T_0(iy)}{L^2(\rr d)\to L^2(\rr d)} &<\infty ,
\\[1ex]
\sup _{y\in \mathbf R}\nm {T_0(1+iy)}
{L^2(\rr d)\to \maclH ^2_{[\vartheta _c]}(\rr d)} &<\infty 
\intertext{and}
\sup _{z\in \overline \Omega}\nm {T_0(z)}{L^2(\rr d)\to L^2(\rr d)}
&\le \sup _{0\le x\le 1} \nm {T}{L^2(\rr d)\to L^2(\rr d)}^x
\end{align*}
for some $c>0$, where $\vartheta _c(\alpha )=e^{c|\alpha |^{\frac 1{2s}}}$.

\par

It now follows from Lemma \ref{Lemma:Interpolation1} and Calderon-Lion's
interpolation theorem
(cf. Theorem IX.20 in \cite{ReSi}) that $T^r$ is continuous from $L^2(\rr d)$ to
$\maclH ^2_{[\vartheta _{rc}]}(\rr d)$. Hence, by duality it follows that
\begin{alignat*}{2}
T^r\, &:\, & L^2(\rr d) & \to \maclH ^2_{[\vartheta _{rc}]}(\rr d)
\intertext{and}
T^r\, &:\, & \maclH ^2_{[1/\vartheta _{rc}]}(\rr d) & \to L^2(\rr d)
\intertext{are continuous. Hence, by interpolation we obtain that}
T^r\, &:\, & \maclH ^2_{[1/\vartheta _{rc/2}]}(\rr d)
& \to \maclH ^2_{[\vartheta _{rc/2}]}(\rr d)
\end{alignat*}
is continuous (cf. Remark \ref{Rem:GelfandTripp}). The result now follows from
$$
\maclH _s(\rr d) = \bigcup _{r>0} \maclH ^2_{[\vartheta _{r}]}(\rr d)
\quad \text{and}\quad
\maclH _s'(\rr d) = \bigcap _{r>0} \maclH ^2_{[1/\vartheta _{r}]}(\rr d).
\qedhere
$$
\end{proof}

\par

We also need the following characterization of Pilipovi{\'c} kernels.

\par

\begin{lemma}\label{Lemma:PilKernelChar1}
Let $p, q \in (0, \infty]$, $\omega \in \mascP (\rr { 2d _2} \times \rr {2d_1})$,
and let $s>0$, $K\in \maclH '_0 (\rr {d _2} \times \rr {d_1})$,
$$
H_1=|x_1|^2-\Delta _{x_1},
\quad %\text{and} 
H_2=|x_2|^2-\Delta _{x_2},\qquad x = (x_2,x_1)\in \rr {d_2}\times
\rr {d_1}.
$$
Also let $H=H_2+H_1$ be the Harmonic oscillator on $\rr {d_2}\times \rr {d_1}$.
Then the following conditions are equivalent:
\begin{enumerate}
\item  $K\in \maclH _s(\rr {d_2}\times \rr {d_1})$ \quad
($K\in \maclH _{0,s}(\rr {d_2}\times \rr {d_1})$);

\vrum

\item for some $h>0$ \quad (for every $h>0$) it holds
\begin{equation}\label{Eq:PilSpaceHarmOscEst}
\| H ^N K\| _{L ^2} \lesssim h ^N (N !) ^{2s}, \qquad N \geq 0,
\end{equation}

\vrum 

\item for some $h>0$ \quad (for every $h>0$) it holds
\begin{equation}\label{Eq:PilSpaceHarmOscSplitEst}
\| H_1^{N_1}H_2^{N_2}K \| _{L ^2}\lesssim
h^{(N_1+N_2)} (N_1!N_2!)^{2s},\qquad N _1, N _2\geq 0.
\end{equation}

\vrum

\item for some $h>0$ \quad (for every $h>0$) it holds
\begin{equation}\label{Eq:PilSpaceHarmOscSplitEstM}
\| H_1^{N_1}H_2^{N_2}K \| _{M ^{p, q} _{(\omega)}}\lesssim
h^{(N_1+N_2)} (N_1!N_2!)^{2s},\qquad N _1 \geq N _{0, 1}, N _2 \geq N _{0,2}.
\end{equation}
\end{enumerate}
\end{lemma}

\par

\begin{proof}
The assertion (1) and (2) are equivalent in view of \cite[Proposition 4.1]{To13}.
We prove (2) and (3) are equivalent.
Assume that \eqref{Eq:PilSpaceHarmOscEst} holds.
Since $K \in \maclH '_0 (\rr {d _2} \times \rr {d_1})$,
it follws that 
$$
K = \sum _{\alpha _1, \alpha _2} c_\alpha (K) h _{\alpha _1}
\otimes h _{\alpha _2}, \quad  h _{\alpha _j},\ j=1, 2 %\ \text{Hermite functions}
$$
and that the Hermite coefficients of $K$ satisfies
$$
|c_\alpha (K)| \lesssim e^{-\frac 1h|\alpha |^{\frac 1{2s}}},
\qquad \alpha =(\alpha _2,\alpha _1)\in \nn {d_2}\times \nn {d_1}.
$$
By parseval's inequality we obtain
\begin{multline*}
\nm {H_1^{N_1}H_2^{N_2}K}{L^2 }
\\[1ex]
\leq 
\left (
\sum _{\alpha _1\in \nn {d_1}}\sum _{\alpha _2\in \nn {d_2}}
(2|\alpha _1|+d_1)^{N_1}(2|\alpha _2|+d_2)^{N_2}
e^{-\frac 1{h}|\alpha |^{\frac 1{2s}}}
\right ) ^{1/2}
\\[1ex]
\le I_1\cdot I_2,
\end{multline*}
where
$$
I_k =  \sum _{\alpha _j \in \nn {d_k}}
(|\alpha _j |+d_k)^{N_k}
e^{-\frac 1{h_0}|\alpha _j |^{\frac 1{2s}}}
$$
with $h_0=ch$, for some constant $c>0$ which only depends on $s$.

By Lemma 4.7 in \cite{To13} and its proof we get
$$
I_k \lesssim (3(4sh_0)^{2s})^{N_k}(N_k!)^{2s}
= (3(4sch)^{2s})^{N_k}(N_k!)^{2s},
$$
and a combination of these estimates shows that (2)
implies (3).

\par

Assume that \eqref{Eq:PilSpaceHarmOscSplitEst} holds instead. Then
\begin{multline*}
\nm {H^N K}{L^ 2} = \nm {(H_1+H_2)^N K}{L^ 2}
\le
\sum _{k=0}^N{N \choose k}\nm {H_1^{N-k}H_2^k K}{L^ 2}
\\[1ex]
\lesssim h^N\sum _{k=0}^N{N \choose k} ((N-k)!k!)^{2s}
\le 
h^N(N!)^{2s}\sum _{k=0}^N{N \choose k} = (2h)^N(N!)^{2s},
\end{multline*}
and it follows that (3) implies (2).

\par

Now we prove the equvialence between (3) and (4).
First we show that 
\begin{equation} \label{Eq:PilSpaceHarmOscSplitEstMx1}
\|H ^N _1 K \| _{M ^{p, q} _{(\omega)}}\lesssim
h ^N (N !) ^{2s}, 
\\
\qquad N \geq N _0
\end{equation}
is independent on $N _0$ and $\omega$ when $p, q \geq 1$.
If \eqref{Eq:PilSpaceHarmOscSplitEstMx1} is true for $N_0=0$,
then it is also true for $N_0>0$.
If $0\le N\le N_0$, $N_1=N_0-N\ge 0$ and 
\eqref{Eq:PilSpaceHarmOscSplitEstMx1} holds for some $N _0 \geq 0$,
then by the fact that 
\begin{gather}
H^N _1 \, :\, M^{p,q}_{(v_N\omega )}(\rr {d _2} \times \rr {d _1})\to
M^{p,q}_{(\omega )}(\rr {d _2} \times \rr {d _1}),\label{HarmonicOscModMap}
\intertext{with}
 v_N(x _1,\xi _1, x _2,\xi _2 ) = C(1+|x _1|^2+|\xi _1 |^2)^N,\notag
\end{gather}
%%
% where $C$ is denpends on $x _2$ and $\xi _2$,
is a homeomorphism
$H ^N _1$ %and its inverse are continuous and bijective
(cf. e.{\,}g. \cite[Theorem 3.10]{SiTo}), it follows that  
$$
\nm {H^N _1 K }{M^{p,q}_{(\omega )}} \lesssim
\nm {H^{N_0} _1 K }{M^{p,q}_{(\omega /v_{N_1})}}
\lesssim \nm {H^{N_0} _1 K}{M^{p,q}_{(\omega )}}<\infty ,
$$
and \eqref{Eq:PilSpaceHarmOscSplitEstMx1} holds for $N _0 = 0$.
This shows that \eqref{Eq:PilSpaceHarmOscSplitEstMx1} is 
independent of $N _0 \geq 0$ when $p, q \geq 1$.

\par

Let $\omega \in \mascP (\rr {2 d _2} \times \rr {2d _1} )$.
Then there exists an integer $N _0 \geq 0$ such that 
$$
(v_{N_0}) ^{-1} 
\lesssim \omega  
\lesssim v_{N_0}  ,
$$
and then
%%%
%\begin{equation}\label{NormArrayCond}
%\nm K {M^{p,q}_{(\omega _2/v_{N_0})}}\lesssim \nm K {M^{p,q}_{(\omega _{11}\omega _2)}},
%\nm K{M^{p,q}_{(\omega _{12}\omega _2)}}
%\lesssim \nm K{M^{p,q}_{(v_{N_0}\omega _2)}}.
%\end{equation}
%%%
%%
\begin{equation}\label{NormArrayCond}
\nm K {M^{p,q}_{( 1 /v_{N_0})}}\lesssim \nm K {M^{p,q}_{(\omega )}}
%\nm K {M^{p,q}_{(\omega _{2})}}
\lesssim \nm K {M^{p,q}_{(v_{N_0})}}.
\end{equation}
Hence the stated invariance follows if we prove that \eqref{Eq:PilSpaceHarmOscSplitEstMx1}
holds for $\omega =v_{N_0}$, if it is true for $\omega =1/v_{N_0}$.

\par

Therefore, assume that \eqref{Eq:PilSpaceHarmOscSplitEstMx1} holds for $\omega  =1/v_{N_0}$.
If $N\ge 2N_0$, then the bijectivity of \eqref{HarmonicOscModMap} gives
%%%
%\begin{multline*}
%\frac {\nm {H^N _1 K}{M^{p,q}_{(v_{N_0}\omega _2)}}}{h^N(N!)^{2s}}
%\lesssim
%\frac {\nm {H^{N+2N_0} _1 K}{M^{p,q}_{(\omega _2/v_{N_0})}}}{h^N(N!)^{2s}}
%\\[1ex]
%=
%h^{2N_0}{{N+2N_0}\choose {2N_0}}^{2s}((2N_0)!)^{2s}
%\frac {\nm {H^{N+2N_0} _1 K}{M^{p,q}_{(\omega _2/v_{N_0})}}}{h^{N+2N_0}((N+2N_0)!)^{2s}}
%\\[1ex]
%\asymp
%{{N+2N_0}\choose {2N_0}}^{2s}
%\frac {\nm {H^{N+2N_0} _1 K }{M^{p,q}_{(\omega _2/v_{N_0})}}}{h^{N+2N_0}((N+2N_0)!)^{2s}}
%\lesssim
%\frac {\nm {H^{N+2N_0} _1 K}{M^{p,q}_{(\omega _2/v_{N_0})}}}{h_1^{N+2N_0}((N+2N_0)!)^{2s}},
%\end{multline*}
%%%
%%
\begin{multline*}
\frac {\nm {H_1^N K}{M^{p,q}_{(v_{N_0})}}}{h^N(N!)^{2s}}
\lesssim
\frac {\nm {H_1^{N+2N_0}K}{M^{p,q}_{(1/v_{N_0})}}}{h^N(N!)^{2s}}
\\[1ex]
=
h^{2N_0}{{N+2N_0}\choose {2N_0}}^{2s}((2N_0)!)^{2s}
\frac {\nm {H_1^{N+2N_0}K}{M^{p,q}_{(1/v_{N_0})}}}{h^{N+2N_0}((N+2N_0)!)^{2s}}
\\[1ex]
\asymp
{{N+2N_0}\choose {2N_0}}^{2s}
\frac {\nm {H_1^{N+2N_0}K}{M^{p,q}_{(1/v_{N_0})}}}{h^{N+2N_0}((N+2N_0)!)^{2s}}
\lesssim
\frac {\nm {H_1^{N+2N_0}K}{M^{p,q}_{(1/v_{N_0})}}}{h_1^{N+2N_0}((N+2N_0)!)^{2s}},
\end{multline*}
where $h_1=\frac h{4^s}$.
This shows that \eqref{Eq:PilSpaceHarmOscSplitEstMx1} is independent of 
$\omega$ in the case $p, q \geq 1$.

\par

By repeating these arguments, it follows that \eqref{Eq:PilSpaceHarmOscSplitEstM} is independent
of $N _{0,1}$, $N _{0, 2}$, $\omega$ and $p,q\in [1,\infty ]$.
For general
$p,q\in (0,\infty ]$, the invariance of \eqref{Eq:PilSpaceHarmOscSplitEstM}
with respect to $N _{0,1}$, $N _{0, 2}$, $\omega$, $p$ and $q$, is now a
consequence of the embeddings
$$
M^\infty _{(v_N\omega )}(\rr {d _2} \times \rr {d _1})
\subseteq M^{p,q} _{(\omega )}(\rr {d _2} \times \rr {d _1})
\subseteq M^\infty _{(\omega )}(\rr {d _2} \times \rr {d _1})
$$
when
$$
\qquad N >\frac d{\min (p,q)}
$$
(see e.{\,}g. \cite{GaSa}).

\par

The equivalence bewteen (3) and (4) now follows
from these invariance properties and the fact that
$$
L^{2} = M^{2,2}.
$$
%the continuous embeddings
%$$
%M^{p_0,q_1}\subseteq L^{p_0}\subseteq M^{p_0,q_2},
%\qquad
%q_1=\min (p_0,p_0'),\quad q_2=\max (p_0,p_0'),
%$$
%which can be found in e.{\,}g. \cite{To8}.
The proof is complete.
\end{proof}

\par

\begin{proof}[Proofs of Theorems \ref{Thm:KernSchwSpaceChar} and
\ref{Thm:KernPilSpaceChar}]
We only prove Theorem \ref{Thm:KernPilSpaceChar} and then in the Roumieu
case. The other cases follow by similar arguments and are left for the reader.

\par

(1) Assume that $K\in \maclH _s (\rr {d_2}\times \rr {d_1})$. By polar decomposition
we have
\begin{alignat*}{2}
K(x_2,x_1) &= \sum _{j=1}^\infty \lambda _{0,j}g_j(x_2)
\overline{f_j(x_1)},&
\quad x_1 &\in \rr {d_1},\ x_2\in \rr {d_2},
\intertext{where $\lambda _{0,j}\ge 0$ are the singular values of $T$,
$\{ f_j\} _{j=1}^\infty \in \ON (L^2(\rr {d_1}))$ and
$\{ g_j\} _{j=1}^\infty \in \ON (L^2(\rr {d_2}))$. Now let $K_1$ and $K_2$
be the kernels of $T_1\equiv (T^*\circ T)^{\frac 14}$ and $T_2\equiv
(T\circ T^*)^{\frac 14}$, respectively. Then}
K_1(x_2,x_1) &= \sum _{j=1}^\infty \sqrt {\lambda _{0,j}}\,  f_j(x_2)
\overline{f_j(x_1)},& \quad
x_1,x_2 &\in \rr {d_1}
\intertext{and}
K_2(x_2,x_1) &= \sum _{j=1}^\infty \sqrt {\lambda _{0,j}}\, g_j(x_2)
\overline{g_j(x_1)},& \quad
x_1,x_2 &\in \rr {d_2}.
\end{alignat*}

\par

By Theorem \ref{thmSchattGSkernels} we get
%
%a straight-forward application of Lemma \ref{Lemma:PilKernelChar1}
%it follows that the kernels of $T^*\circ T$ and $T\circ T^*$ belong to
%$\maclH _s (\rr {d_1}\times \rr {d_1})$ and $\maclH _s (\rr {d_2}\times \rr {d_2})$,
%respectively. Hence, $K_k\in \maclH _s (\rr {d_k}\times \rr {d_k})$ by
%Proposition \ref{Prop:SelfAdj},
%since $T_k$ are positive semi-definite on $L^2(\rr {d_k})$, $k=1,2$.
%Since $\lambda _j$ are the singular values of both $T_1$ and $T_2$,
%Theorem 3.1 in \cite{ChSiTo} gives
%$$
%\lambda _{0,j} \lesssim e^{-rj^{\frac 1{2d_1}}}
%\quad \text{and}\quad
%\lambda _{0,j} \lesssim e^{-rj^{\frac 1{2d_2}}},
%$$
%for some constant $r>0$, that is,
%%
\begin{equation}\label{Eq:ImprSingValEst}
\lambda _{0,j} \lesssim e^{-rj^{\frac 1{2d}}}
\end{equation}
for some constant $r>0$.

\par

Since $K_1\in \maclH _s (\rr {d_1}\times \rr {d_1})$, Lemma \ref{Lemma:PilKernelChar1}
gives
$$
\left (\sum _{j=1}^\infty \lambda _{0,j} \nm {H^Nf_j}{L^2}^4\right )^{\frac 12}
= \nm {H_1^NH_2^NK_1}{\operatorname{Tr}} \le \nm {H_1^NH_2^NK_1}{M^{1,1}}
\le
h^N(N!)^{4s},
$$
where $\nm \cdo{\operatorname{Tr}}$ is the trace-class norm. Here we have
identified operators with their kernels, and used the fact that operators with kernels
in $M^{1,1}(\rr {2d})$ are of trace-class (cf. \cite{GH1,To13}).
Hence,
$$
\lambda _j^{\frac 14}\nm {H^Nf_j}{L^2}\lesssim h_0^N(N!)^{2s},
$$
where $h_0=\sqrt h$. Hence, if $f_{1,j} = \lambda _{0,j}^{\frac 13}f_j$
we obtain
$$
\nm {H^Nf_{1,j}}{L^2}\lesssim \lambda _{0,j}^{\frac 1{12}}h_0^N(N!)^{2s}
\lesssim
e^{-rj^{\frac 1{2ds}}}h_0^N(N!)^{2s},
$$
for some $r>0$. By considering $K_2$ instead of $K_1$ and letting
$f_{2,j}= \lambda _{0,j}^{\frac 13}g_j$, the same computations
give
$$
\nm {H^Nf_{2,j}}{L^2}\lesssim e^{-rj^{\frac 1{2ds}}}h_0^N(N!)^{2s}
$$
as well.

\par

The assertion now follows if we let $\lambda _j = \lambda _{0,j}^{\frac 13}$.

(2) By the assumptions and Cauchy-Schwartz inequality, we obtain
\begin{multline*}
\|H ^{N _1} _1 H ^{N _2} _2 K \| _{L ^2} 
= \|H ^{N _1} _1 H ^{N _2} _2 (\sum \lambda _j f _{1,j} \otimes f _{2, j}) \| _{L ^2}
\\
= \left( \iint _{\rr {d _2} \times \rr {d _1}}  
\Big| \sum \lambda _j H ^{N _1} _1f _{1,j} \otimes H ^{N _2} _2 f _{2, j}\Big| ^2dx_1 \ dx _2 \right) ^{1/2}
\\
\leq  \left( \iint _{\rr {d _2} \times \rr {d _1}}  
\sum \Big| \lambda _j \Big| ^2 
\sum \Big |H ^{N _1} _1f _{1,j} \otimes H ^{N _2} _2 f _{2, j}\Big| ^2dx_1 \ dx _2 \right) ^{1/2}
\\
\leq  \sum \lambda _j  \left( \sum \|H ^{N _1} _1 f_{1, j} \| ^2 _{L ^2}
\|H ^{N _2} _2 f_{2, j} \| ^2 _{L ^2}\right) ^{1/2}
\\
\lesssim h ^{N _1 + N _2}(N _1! N _2 !) ^{2s} \sum e ^{- r j^{1/{2ds}}}
\\
\lesssim h ^{N _1 + N _2}(N _1! N _2 !) ^{2s}.\qedhere
\end{multline*}
\end{proof}

\par

\end{document}